\documentclass[12pt]{amsart}
 \usepackage{amsmath}
\usepackage{amssymb} 
\usepackage{amsfonts}
     \usepackage{xcolor}
    \makeatletter
     \def\section{\@startsection{section}{1}%
     \z@{.7\linespacing\@plus\linespacing}{.5\linespacing}%
     {\bfseries \normalfont\scshape
     \centering
     }}
     \def\@secnumfont{\bfseries}
     \makeatother

   \newtheorem{theorem}{Theorem}[section]
\newtheorem{lemma}[theorem]{Lemma}
\newtheorem{corollary}[theorem]{Corollary}
\newtheorem{proposition}[theorem]{Proposition}

\theoremstyle{definition}
\newtheorem{definition}[theorem]{Definition}

\theoremstyle{remark}
\newtheorem{remark}[theorem]{\bf Remark}
\newtheorem{remarks}[theorem]{\bf Remarks}

\numberwithin{equation}{section}

\def\beq{\begin{equation}}
\def\eeq{\end{equation}}
  
\def\ben{\begin{eqnarray}}
\def\een{\end{eqnarray}}

\def \a{{\alpha}}
\def \b{{\beta}}
\def \D{{\Delta}}
 \def \d{{\delta}}
\def \e{{\varepsilon}}
 \def \g{{\gamma}}

\def \k{{\kappa}}
\def \l{{\lambda}}

\def \o{{\omega}}
\def \O{{\Omega}}
\def \p{{\varphi}}

 \def \m{{\mu}}
\def \s{{\sigma}}
 
  \def \A{{\mathcal A}}

\def \qq{{\qquad}}

\def \X{{\mathcal X}}

  \def \noi{{\noindent}}

\def\E{{\mathbb E}}

\def\P{{\mathbb P}}

\def\R{{\mathbb R}}
\def \dd{{\rm d}}
 
\def\Z{{\mathbb Z}}
 
\def\Q{{\mathbb Q}}

\def\N{{\mathbb N}}

at  6 pt
\scrollmode

  \font\ggum= cmbx10 at 10  pt
 at 9,5 pt  \scrollmode

\font\sevenrm= cmr10 at 7,3 pt
 at 9,3 pt
\font\tenrm= cmr10 at 12,5 pt

\scrollmode

\def\ddate {\sevenrm \ifcase\month\or January\or
February\or March\or April\or May\or June\or July\or
August\or September\or October\or November\or December\fi\! {\the\day}, \!{\sevenrm\the\year}}
  
\begin{document}
   %\title[A general decoupling inequality for finite Gaussian vectors]{ \rm { \bf   A general decoupling inequality for finite Gaussian vectors}\vskip 1 pt {\sevenrm (\ddate)}} 

 % \title[ ]{ \rm { \bf An   approach using Brownian motion for estimating the Rate of convergence in the Fourier inversion formula}\vskip 1 pt (\ddate)}
%  \title[moderate deviations  of   suprema of Gaussian  polynomials]{ \rm { \bf  Semi-asymptotic bounds for   moderate deviations  of  suprema of  
% trigonometric and almost periodic 
%  Gaussian  polynomials
 %}\vskip 1 pt {\sevenrm (\ddate)}}%  \title[ Semi-asymptotic bounds for   moderate deviations  of   Gaussian  processes]{ \rm { \bf  Semi-asymptotic bounds for   moderate deviations  of  suprema of  
 % parametrized families ofparametrized  Gaussian 
 %trigonometric  polynomials and  faded Ornstein-Uhlenbeck processes}\vskip 1 pt {\sevenrm (\ddate)}} 
 \title[{\sevenrm moderate deviations  of   suprema of Gaussian  polynomials}]{{\ggum 
 %Semi-asymptotic bounds for   
 Moderate deviations  of  suprema of  
   Gaussian  processes \\ An approximation criterion
   %, and an elementary multivariate Berry-Esseen inequality
   }
%   \\  {\sevenrm (\ddate)}
   %{\sevenrm (Preliminary version)}
   }
    \title[{\sevenrm moderate deviations  of   suprema of Gaussian  polynomials}]{{\ggum 
 %Semi-asymptotic bounds for   
  Moderate deviations  of  suprema of  
    Gaussian  processes \\ A cyclic approximation criterion
   %, and an elementary multivariate Berry-Esseen inequality
   }
%   \\  {\sevenrm (\ddate)}
   %{\sevenrm (Preliminary version)}
   }
   \author{Michel J.\,G. Weber}  \address{Michel Weber: IRMA, UMR 7501, 10  rue du G\'en\'eral Zimmer, 67084
 Strasbourg Cedex, France}
  \email{michel.weber@math.unistra.fr}

   % \subjclass[2010]{Primary:\,42A38 ; Secondary:\,60G50}
      \begin{abstract}  
We study  moderate deviations  of     suprema of     parametrized sequences of  sample bounded Gaussian processes $\{X _x(t), t\in T _x\}$, and first present recent sharp bounds in simple cases. 
   Let $\{L_k, k\ge 1\}$ be any increasing unbounded sequence of  positive reals and $(a_k)_{k\ge 1}$ a sequence of real numbers such that $A_x = \sum_{k\le x} a_k^2\uparrow \infty$  with $x$,  and let     \begin{eqnarray*} X_{y,x}(u)  =  \sum_{y\le k\le x} a_k \big( g_k \cos L_ku+ g'_k\sin L_k u\big),\qq   x>y\ge 1,\quad u\in\R, \end{eqnarray*}  
    %    Let      \begin{eqnarray*} X_x(t) &=& \sum_{k\le x} a_k \,\big( g_k \cos 2\pi k   \,t+  g'_k\sin 2\pi  k  \, t\, \big), \quad 0\le t\le 1, \ x>0, 
%\end{eqnarray*}
where   $(g_k)_{k\ge 1}$, $(g'_k)_{k\ge 1}$ are two independent sequences of i.i.d. $\mathcal N(0,1)$ distributed random variables. We first study the periodic case 
    $L_k\equiv k$. Assume   that   $A(x)\sim \log\log x$, $x\to \infty$ and $B=\sum_{k\ge 1} a_k^4<\infty$. Let $0<\eta< 1$.  We prove for instance that
 there exists   an absolute constant  $C$  such that for all $x$ large enough,
\begin{align*}
  \P\Big\{ \sup_{0\le t\le 1}
  X_{1,x}(t) \le\sqrt{2\eta (\log\log x)(\log \log\log x)}\Big\}
 \ \le  \ e^{-\,\frac{C (\log\log x)^{1-\eta}}{  \sqrt{8\eta (B+1)(\log \log\log x)}}}.
 \end{align*} 
In the almost periodic case,   %  We   establish an approximation theorem  of   moderate deviations of almost periodic Gaussian polynomials by Gaussian polynomials with $\Q$-frequencies, and show  that the error term has an exponential decay. 
we prove an approximation theorem. We introduce a   modulable diophantine approximation.   Let $N_k\ge k, \  k\ge 1$ be a non-decreasing unbounded test   sequence of  positive  integers, and let 
$\ell(k)=\ell(N_k,k)=\frac{1}{N_k}\big\lfloor N_kL_k\big\rfloor$, so that 
 $  \big|\ell( k) -L_k\big|\le \frac{1}{N_k}$,  $k\ge 1$.
  Put  for any interval $I$,
 $ \k (I)=  \#\{\k:  [N_{\k-1}, N_\k[\subset I\}$.
 %  The test sequence $\{N_k, k\ge 1\}$ needs not necessarily grows fast  with $k$.  
 %Consider  the nearby  (cyclic) stationary Gaussian processes  
%  \begin{equation}X^\perp_{ y,x}(u) = \sum_{   y\le  k\le x} a_k  \big( g_k \cos ( \ell( k)    u) \big) +& g'_k\sin  ( \ell( k)   u)  \big),\qq  u\in\R.
% \end{equation} 
%The setting considered is of course too large.
We   prove an   approximation theorem  by Gaussian polynomials with $\Q$-frequencies, and exponentially decaying error term:  for any reals   $\Theta_{y,x}>0$, $1\le y\le x$, $U\ge 1$, $0<h<  \Theta_{y,x}$,
\begin{equation*}  \P\Big\{   \sup_{1\le u\le U} X_{y,x}(u)\le \Theta_{y,x} -h  \Big\}   \,  \le \, 
\P\Big\{ \sup_{1\le u\le U} X^\perp_{y,x}(u) \le \Theta_{y,x}   \Big\} +2\, \exp\Big\{\frac{- C\,h^2 }{     \D^2   \log \k([1,U]) } \Big\},
  \end{equation*}
where
$X^\perp_{ y,x}(u) = \sum_{   y\le  k\le x} a_k  \big( g_k \cos ( \ell( k)    u) \big) +  g'_k\sin  ( \ell( k)   u)  \big)$, $u\in\R$,
and 
%  $\D_1=\D$, if $1\le y\le U$ and $\D_1=\D'$, if $1\le U\le y$.      
$$\D\,=\,
\sqrt{   \sum_{   y\le k\le x} \frac{1 }{N^2_k  } }   \sqrt{   \sum_{   y\le k\le x}a_k^2 } +
             \sum_{ y\le  k<  \k \atop N_\k \le U}   |a_k|
   +    \sup_{ y\le N_\k\le U }\ N_\k 
    \sqrt{ \sum_{   \k\le k\le x} \frac{  1}{N_k^{2 } }       }\, 
    \sqrt{    \sum_{   \k\le k\le x}   |a_k|^2    },$$  
 if $1\le y\le U$, and $\D\,=\,U \sqrt{ \sum_{   y\le k\le x}\frac{1 }{N^2_k  } }\sqrt{ \sum_{   y\le k\le x}a_k^2\ }$, if $1\le U\le y$.
 %\begin{equation*}  \D\,=\,  \begin{cases}\ \  y  \   \big( \displaystyle{\sum_{   y\le k\le x} \frac{1 }{N^2_k  }\big)^{1/2}\big( \sum_{   y\le k\le x}a_k^2\big)^{1/2} +
  %           \sum_{ y\le  k<  \k\atop N_\k \le U}   |a_k|}
 % \cr  \    \qquad   +   \displaystyle{\sup_{\k:\,y\le N_\k\le U }\ N_\k  \big( \sum_{   \k\le k\le x} \frac{  1}{N_k^{2 } }      \big)^{1/2}\,\Big(\sum_{   \k\le k\le x}   |a_k|^2    \big)^{1/2}} \quad & \hbox{if $1\le y\le U$,}
%  \cr \  \ U \Big( \sum_{   y\le k\le x}\frac{1 }{N^2_k  }\Big)^{1/2}\Big( \sum_{   y\le k\le x}a_k^2\Big)^{1/2}
%  \quad & \hbox{if $1\le U\le y$.}
%\end{cases} 
% \end{equation*}
  %The study of the deviation of  $\sup_{1\le u\le U} X^\perp_{y,x}

%\noi(iii) (Lattice behavior)
 Finally we study for general non-vanishing coefficient sequences, the behavior along lattices of almost periodic Gaussian 
polynomials with linearly independent frequencies, and  use a lattice localized version of Kronecker's theorem.
\end{abstract}
\maketitle

%\tableofcontents

%%%%%%%%%%%%%%%%%%%%%%%%%%%%%%%%%%%%%%%%%%%%%%%%%%%%%%%%%%%%%%%%%%%%%%%%%%%%%%%%%%%%%%%%%%%%%%

  \section{Introduction-Results}\label{s1}
% $$ X{\hskip -6 pt  \raisebox{7pt}{\neg}} $$
 We study  moderate deviations  of     suprema of  sample bounded  $x$-parametrized Gaussian processes $\{X _x(t), t\in T _x\}$, namely    bounds of type       
 \beq  \label{db.type} \P\big\{ \sup_{t\in T _x}
 X _x(t) \le \Theta _x \big\}\le e^{-\Upsilon _x} ,\qq \qquad \Theta _x, \Upsilon _x\to +\infty,\quad x\to  +\infty.\eeq

     %, a bound of type   below can be established and valid for all $x$ large enough,   $$   \P\Big\{ \sup_{0\le t\le 1}
%\widetilde X_{y,x}(t) \le \Theta_x \Big\}\le e^{-\Upsilon_x(x)}.$$   The general underlying problem is the one of finding sharp bounds 
%    \beq\label{gp.bound}
 %   $$ \P\Big\{\sup_{j=1}^n X_j\le \Theta \Big\},$$
 %   \eeq
%where $X= \{X_1, \ldots, X_n\}$ is a centered Gaussian vector.  
   Let $$
\Phi(x)  ={1\over \sqrt{2\pi}} \int_\infty^x e^{-t^2/2} \dd t, \qq \Psi(x) = 1-\Phi(x).
 $$
 
    In the simplest case when $X= \{X_1, \ldots, X_n\}$, $n\ge 2$,  is a centered Gaussian vector, such that  for some  $0<\l <1$,  \beq\label{assumption}\hbox{$\E X^2_u=1 $, \qq  $\E X_uX_v\le \l$ , \qq   $u\neq v$,}
\eeq
 a sharp deviation bound  with Gumbel's type version holds.  
 %Whereas  this one deals with convergence in law  of sequences of partial maxima, we are   interested in obtaining individual estimates for  these ones. Further the convergence in law is usually obtained by applying derivation's theorems.
  \begin{theorem}[Weber \cite{W1a},\,Th.\,1.1]
\label{Ac1c} 
{ \rm (i)} For   any positive real   $\Theta$,  and any $n\ge 2$,
  \begin{eqnarray}\P\Big\{ \sup_{i=1}^n X_i\le \Theta\Big\}  &\le &\Big(1+\frac{  \l n  }{   1-\l }\Big)^{(n-1)/2}  \ \Phi \Big(\frac{\Theta}{\sqrt{ 1+\l(n-1) }}\Big)^n   .
 \end{eqnarray}
 {\rm (ii)}   Set
$b_n= \big( \log \frac{n^2}{4\pi \log n} \big)^{1/2}$.  
  Given any positive real $\e$, for $n$ sufficiently large, $n\ge n(\e)$,  and all $x$ such that  $x\ge -b_n^2$,
\begin{eqnarray*}\P\Big\{ \sup_{i=1}^n X_i\le \Big({\frac{1}{b_n}x+ b_n}\Big) {\sqrt{ 1+\l(n-1)}} \Big\}     \,\le \, \Big(1+\frac{  \l n  }{   1-\l }\Big)^{(n-1)/2}  e^{- e^{-x} (1-\e)   } .
 \end{eqnarray*}  \end{theorem} 
  We refer to     Berman \cite[Ch.\,9]{Be}, concerning Gumbel law, extreme value theory, namely  convergence in law results for sequences of partial maxima, usually obtained by applying derivation   theorems.

\vskip 3 pt Let assumption \eqref{assumption} be slightly strengthened as follows: for some real $0<\l<1$,
  \beq\label{assumption1}\hbox{$\E X^2_u=1 $, \qq  $|\E X_uX_v|\le \l$, \qq   $u\neq v$.}
\eeq

\begin{theorem}\label{ass1th} Let $0<\eta<1$ be fixed. There exist positive constants $C_\eta, C'_\eta$ depending on $\eta$ only, such that for any centered Gaussian vector $X=\{X_1,\ldots, X_n\}$, $n\ge 3$, satisfying  assumption \eqref{assumption1}   for some positive real $ \l  $ such that  $ \l\le \eta/2n  $,  we have 
\begin{eqnarray*}
 \P\Big\{ \sup_{i=1}^n X_i\le \sqrt{    2  \log n -2 \log\log n   -   \eta  (\log n)/      n }\Big\}
 &\le & C_\eta\, e^{- C'_\eta \sqrt{ \log n  }  }.
\end{eqnarray*}
\end{theorem}
%\begin{eqnarray*}
% \P\Big\{ \sup_{i=1}^n X_i\le \sqrt{    2(1-\l)\log \frac{n}{   \log n}   }\Big\}
% &\le & C_\eta\, e^{- C'_\eta \sqrt{ \log n  }  }.
%\end{eqnarray*}$$ \qquad 2\l \log \frac{n}{    \log n}\le     \frac{ \eta  \log n}{     n} \qq \sqrt{    2 \log \frac{n}{   \log n} - 2\l \log \frac{n}{   \log n}}\ge \sqrt{    2 \log \frac{n}{   \log n} - \frac{ \eta  \log n}{     n} }$$
%\begin{eqnarray*}
% \P\Big\{ \sup_{i=1}^n X_i\le \sqrt{    2(1-\l)\log \frac{n}{   \log n}   }\Big\}
 %&\le & C_\eta\, e^{- C'_\eta \sqrt{ \log n  }  }.
%\end{eqnarray*} 
\vskip 3 pt 
Let further be given reals, $0<\l\le u<1$.  Theorem \ref{Ac1c} is  generalized in  \cite{W1a} by considering       
centered Gaussian processes $\{X(t), t\in T\}$ with unit variance  satisfying  the following property: 
\vskip 3 pt {\it There exists a finite partition $\{T_j,j\in J\}$ of $T$, with $\#\{T_j\}=k$ for each $j$, $\#\{T \}=Nk$, such that}
\ben   \begin{cases} \E X(s)X(t)  \le u  \qquad \forall {s\neq t\in T_j}, \ \forall j\in J, \cr&\cr
\E X(s)X(t) \le  \l  \qquad  \forall {s \in T_i, \forall t\in T_j}, \ \forall i,j\in J, \, i\neq j.
\end{cases} 
\een

This  situation is often met in  Gaussian sample paths study. %We find in \cite{W1a} the following bound.
\begin{theorem}\label{gen.case.th} {\rm (i)}  For   any positive real   $\Theta$, 
\begin{eqnarray*}\P\Big\{ \sup_{t\in T}  X(t)\le \Theta\Big\}&\le  &\int_{\sup_{i=1}^{Nk} x_i\le \Theta}  \frac{e^{-  \frac{ 1}{2 }  Q(\l,\m)(\underline x)   }    }{(2\pi)^{Nk/2}[ (  1- u)^{(k-1)N}\big(  1+u( k-1 )\big))^N]^{1/2}] }\,\dd \underline x  ,
 \end{eqnarray*}
where for any $\underline x\in \R^{Nk}$,
\begin{eqnarray*}Q(\l,\m)(\underline x)  &=& - \frac{   \l  }{(1-u+ k(u-\l))\,( (1-u)+Nku +(N-1)k\l )}\Big(\sum_{ v=1}^{Nk} x_{v}\Big)^2  
\cr & &\   -\frac{ (u-\l)  }{(1-u)(1-u+ k(u-\l))}\sum_{j=0}^{N-1}\Big(\sum_{ l=1 }^k x_{jk+ l}\Big)^2 +\frac{1}{1-u}\Big(\sum_{ v=1}^{Nk} x^2_{v}\Big).
 \end{eqnarray*}

{\rm (ii)} Let 
  $$\b(\l,u) =\Big(  \frac{1}{ 1-u+ k(u-\l) }  \Big) \Big(  \frac{  1-u+Nku-k\l }{    1-u +Nk(u +\l)-k\l  }  \Big). $$
 Assume that $Nu>\l$ and $(k-1)u>1$. Then $0<\b(\l,u)<1$, $\b(\l,u)\asymp\frac{1}{ 1-u+ k(u-\l) }  \frac{   ( N u- \l )}{     (N  u   - \l ) +N\l}$. For each $\underline x\in \R^{Nk}$,   any positive real   $\Theta$, 
\begin{eqnarray*}\P\Big\{ \sup_{t\in T}  X(t)\le \Theta\Big\} &\le &\int_{\sup_{i=1}^n x_i\le \Theta}  \frac{e^{-  \frac{ \b(\l,u)}{2 }    \sum_{   1\le v\le Nk}x_{v}^2  }    }{(2\pi)^{n/2}[ (  1- u)^{(k-1)N}\big(  1+u( k-1 )\big))^N]^{1/2}}\,\dd x
  .
 \end{eqnarray*}
 \end{theorem}\vskip 3 pt  
 Recall Slepian comparison lemma.
   \begin{lemma}[\cite{S}, Lemma\,1,\,p.\,468] Let $Y=\{Y_1,\ldots, Y_n\}$, $X=\{X_1,\ldots, X_n\}$, be two centered Gaussian vectors. Assume that
\begin{equation}
\begin{cases} \E Y_i^2=\E X_i^2 &\qquad 1\le i\le n, \cr
\E Y_iY_j\le\E X_iX_j &\qquad 1\le i,j\le n.
\end{cases}
\end{equation}
Then for any positive real number $x$,
\begin{equation}
 \P\Big\{\sup_{i=1}^n Y_i<x\Big\}\,\le \,\P\Big\{\sup_{i=1}^n X_i<x\Big\}.
\end{equation}
\end{lemma} 
In view of Slepian   lemma, the proof %(as the one of Theorem \ref{Ac1c} too)
 amounts to study matrices of type below.  Let    $R=R(\l)$ be the $k\times k$ matrix whose entries all equal to $\l$, and $C=C(u) $ be the $k\times k$ matrix whose entries off the diagonal are all equal to $u$, and to 1 otherwise. Define the $Nk\times Nk$ matrix composed of $N$ diagonal blocks  $C $ and of  blocks $R $ elsewhere,
 \begin{eqnarray*}& &C(\l,u)\,=\,{ \left[\begin{matrix}
  C  &{}&{}&{}&{}\cr 
   {} &C  &{} &R &{}\cr
  {}&{}  &\ddots &{}&{}
    \cr
 {}       &R  &{}&\ddots& {}
  \cr 
 {}&{}&{}&{} &C \end{matrix} \right]}
  \end{eqnarray*}  
Both Theorems \ref{ass1th}, \ref{gen.case.th} are proved in \cite{W1a} as well. %In view of Slepian comparison lemma, this allows one to derive  bounds of type  \eqref{db.type} for  \vskip 3 pt 
\vskip 15 pt  Otherwise the following typical optimal result    illustrates well our purpose.
\begin{theorem}[\cite{W},\,Th.\,4.1] \label{szego} 
Let  
$\{X_j, j\in \N\}$ be a Gaussian stationary   sequence with spectral  function $F $.  Assume
$F$ is of infinite type (its range  consists of a infinite number of values). Let 
$f$ be the  Radon-Nykodim derivative  of the absolutely continuous part of $F$. Then for   all
$n
$ and
$z>0$,
$$ \Big( 
\int_{-z }^{z } e^{-x^2/2}\frac{\dd x}{\sqrt{2\pi}} \Big)^n\le   \P\Big\{ \sup_{j=1}^n |X_j|\le z \Big\}\le   \Big( 
\int_{-z/\sqrt{G(f)}}^{z/\sqrt{G(f)}} e^{-x^2/2}\frac{\dd x}{\sqrt{2\pi}} \Big)^n
 ,
$$
where
\begin{equation*}  G(\o)  =\begin{cases}\exp\Big\{ \frac1{2\pi}\int_{- \pi}^{ \pi}
\log \o(t) \dd t\Big\} &\quad {\rm if}\  \log \o(t) \ {\rm is\  integrable}\cr 
0 &\quad {\rm otherwise}.
\end{cases}
\end{equation*}
\end{theorem}   
 The second inequality uses determinantal properties of Toeplitz forms (Grenander  and Szeg\"o \cite{GS}, \S1.11, \S 2.2), and   monotonicity property of the solutions of an extremal problem in \cite[p.\,44]{GS},  the first inequality is  Khatri-\v Sid\'ak's.
\vskip 3 pt 
We study \eqref{db.type} for  parametrized families of trigonometric  or almost periodic 
  Gaussian  polynomials.
  % $\{X_x(t), t\in T,x\ge1\}$, and     $\Theta_x, \Upsilon_x $ are tending to infinity with $x$. 
  Throughout   $(g_k)_{k\ge 1}$, $(g'_k)_{k\ge 1}$ (resp.  $(\e_k)_{k\ge 1}$) are two independent sequences of i.i.d., $\mathcal N(0,1)$   distributed   random variables (resp. Rademacher random variables, $\P\{ \e_n =\pm 1\} =1/2$)  defined in a common probability space $(\O,\A,\P)$.  The letter $\X$ (resp.\,$X$)  denotes a centered stationary Gaussian process, but also a parametrized family of centered stationary Gaussian processes  (resp. a centered Gaussian vector).
\subsection{The trigonometrical case.} Consider  for $x\ge y\ge 1$ the families of parametrized Gaussian trigonometrical polynomials 
  \begin{eqnarray}\label{gdsc.yx}\widetilde X_{y,x}(t) &=& \sum_{y\le k\le x} a_k \,\big( g_k \cos 2\pi j_k   \,t+  g'_k\sin 2\pi  j_k  \, t\, \big), \quad 0\le t\le 1,
 \end{eqnarray}
 where $j_k$, $k\ge 1$ are increasing integers.
 % We think of the cases $1\ll y\ll x$ for later use.
  Let
 \begin{equation}\label{A.yx}A(y,x)=\sum_{y\le k\le x}a^2_k, \qq 1\le y\le  x.
 \end{equation}
   We study for which  $\Theta_{y,x}, \Upsilon_{y,x} $ tending to infinity with $x$, a bound of type   below can be established and be valid for all $x$ large enough,   $$   \P\Big\{ \sup_{0\le t\le 1}
\widetilde X_{y,x}(t) \le \Theta_{y,x} \Big\}\le e^{-\Upsilon_{y,x}}.$$
 %\end{remark}  
We prove  \begin{theorem}  \label{semi.asymp.bound.cor} {\rm (1)  ($0<\eta< 1$)}  Assume that for $y, x$ with $2\le y\le x$,   the following condition is satisfied,
\begin{equation}\label{A1.a.yx}\Big(\sum_{y\le k\le x} a_k^4\Big)^{1/2}\le \frac{A(y,x)^{1-\eta}}{\sqrt{\log  A(y,x)}} .
\end{equation}
   Then for  $ 
   % (2\pi \sum_{y\le k\le x} j_ka_k^2)^{-1} \le 
   0<\e \le  1$,
 %such as $A(x)> e$,
\begin{align}\label{semi.asymp.bound.res}
 \P\Big\{ \sup_{0\le t\le \e}
\widetilde X_{y,x}(t) \le \sqrt{2\eta A(y,x)\log A(y,x)}\Big\}
\ \le  \ 
%e^{-\,\frac{C \e\,  A(y,x)^{1-\eta}}{  \sqrt{8\eta  ((\sum_{y\le k\le x} a_k^4)^{1/2}+1 )\log A(y,x)}}}
e^{-\, \frac{C  \e\,  A(y,x)^{1-\eta}}{  \sqrt{ \eta    ((\sum_{y\le k\le x} a_k^4)^{1/2}+1  )\log A(y,x)}}},
\end{align}
where $C$ is an absolute constant. 
 \vskip 2 pt \indent
 If $B=\sum_{k\ge 1} a_k^4<\infty$, then condition \eqref{A1.a.yx} is satisfied and
 % {\color{blue} (for all $x$ large enough, check): plutot pour ces valeurs de $y,x$ ?}
\begin{align*}
 \P\Big\{ \sup_{0\le t\le \e}
 \widetilde X_{y,x}(t) \le\sqrt{2\eta A(y,x)\log A(y,x)}\Big\}
%\ \le  \ e^{-\,\frac{C \e\,  A^{1-\eta}(y,x) }{  \sqrt{8\eta (B+1)\log A(y,x)}}}
\ \le  \ e^{-\,\frac{C \e\,  A^{1-\eta}(y,x) }{  \sqrt{ \eta (B+1)\log A(y,x)}}}.
\end{align*}
{\rm (2)   ($\eta=1$)}. Let $0<V(y,x)<A(y,x)$. Then we have
%  for  $ (2\pi \sum_{y\le k\le x} j_ka_k^2)^{-1}\le \e \le  1$,
$$\P\Big\{ \sup_{0\le t\le \e}\widetilde X_{y,x}(t)\, \le \,  \sqrt{2A(y,x)   \log \Big( \frac{A(y,x)} {V(y,x)}\Big) }\Big\}\,\le \,    e^{-\,\frac{C \e\,V(y,x)}{  \sqrt{   ((\sum_{y\le k\le x} a_k^4)^{1/2}+1 ) \log ( {A(y,x)}/{V(y,x)})}}}.$$ \end{theorem}
  
 \begin{remark}By Cauchy-Schwarz's inequality $\big(\sum_{y\le k\le x} a_k^4\big)^{1/2}\le \big(\sum_{y\le k\le x} a_k^2\big)$. Condition \eqref{A1.a.yx} requires a little more, and that $ A(y,x)$  be  large.   
\end{remark}
   Apply Theorem \ref{semi.asymp.bound.cor} with $y=1$ and let $A(x)=\sum_{k\le x} a_k^2$ so that assumption \eqref{A.yx} translates to 
\begin{equation*} \Big(\sum_{k\le x} a_k^4\Big)^{1/2}=o\big( \frac{A(x)^{1-\eta}}{\sqrt{\log A(x)}}\big).
\end{equation*}
As a consequence we get  
  \begin{corollary}  \label{semi.asymp.bound.a} Let $0<\eta< 1$.   Assume that $A(x)\sim \log\log x$, $x\to \infty$ and $B=\sum_{k\ge 1} a_k^4<\infty$.
 Then there exists   an absolute constant  $C$, such that for all $x$ large enough,
\begin{align*}
  \P\Big\{ \sup_{0\le t\le 1}
  X_x(t) \le\sqrt{2\eta (\log\log x)(\log \log\log x)}\Big\}
 \ \le  \ e^{-\,\frac{C (\log\log x)^{1-\eta}}{  \sqrt{8\eta (B+1)(\log \log\log x)}}}.
 \end{align*} 
 \end{corollary}
  % (announced in \cite[Section\,5]{W3}), %Under this assumption, we have  for all 
 % We use  decoupling inequalities for general  Gaussian vectors, and classical tools   from   Gaussian processes theory.
 % Our approach  is   based on      our recent paper    
%Weber \cite{W3} where        decoupling inequalities for   Gaussian vectors are established,  see    also   Weber \cite{W}, \cite{64}.

% $x$ large enough,
 %such as $A(x)> e$,
%\begin{align}\label{semi.asymp.bound.res.1}
% \P\Big\{ \sup_{0\le t\le 1}
 %X_x(t) \le \sqrt{2\eta A(x)\log A(x)}\Big\}
%\ \le  \ e^{-\,\frac{C A(x)^{1-\eta}}{  \sqrt{8\eta  ((\sum_{k\le x} a_k^4)^{1/2}+1 )\log A(x)}}},
%\end{align}
%where $C$ is an absolute constant. 
 % If $B=\sum_{k\ge 1} a_k^4<\infty$, then condition \eqref{A1} is satisfied and for all $x$ large enough,
%\begin{align*}
 %\P\Big\{ \sup_{0\le t\le 1}
 %X_x(t) \le\sqrt{2\eta A(x)\log A(x)}\Big\}
%\ \le  \ e^{-\,\frac{C A(x)^{1-\eta}}{  \sqrt{8\eta (B+1)\log A(x)}}}.
%\end{align*}
% \end{remark}
%  See section \ref{s2}.   
%The seminal work for these questions is the paper by  Klein, Landau and Shucker \cite{KLS},    which is based  on an analytic inequality due to Brascamp and Lieb   \cite{BL}.   
 \subsection{The almost periodic case.}
%The case of almost periodic 
  %Gaussian  polynomials is far less developed and also quite more   complicated. 
   The same question for 
   %   moderate deviations  of  suprema of  
  almost periodic Gaussian  polynomials %($\R$-frequencies)
   is of quite another 
   %an entirely different 
   order, and is naturally    more   complicated,
  % cannot be treated in general, as in the case of Gaussian trigonometrical   polynomials, and 
  even in known examples, but in the same time, quite more interesting.
  % It is far less developed and also quite more   complicated.   %,  and the theory  exhibits clear limitations. 
  %it is necessary before continuing  to 
  %We briefly place the problem  in the general setting of 
  Let for instance 
   $  X_x(t)=\sum_{1\le k\le x} a_k \big( \e_k  \cos L_kt+ \e'_k\sin L_k t\big)$,  the quantity
  $$\E \,\sup_{t\in  \R} X_x(t) $$
is   estimated for fixed $x$, only in very specific cases.  The optimal estimate of Hal\'asz-Queff\'elec 
$$ \E\, \sup_{t \in \R} \Big|\sum_{n=2}^N
\e_n n^{-\s - it}\Big| \approx 
{N^{1-\s}\over \log N},
$$
is extended in \cite{LW1} and  \cite{W8} to the case $X_N(t)=\sum_{n=1}^N\e_n d(n) n^{- s }$, 
where  $ d  $ is a sub-multiplicative function.  See Queffelec   \cite{Q1}, \cite{Q2}    \cite{LW}, Lifshits and Weber \cite{LW1}, Weber \cite{W7}, \cite{W8}, and  \cite{W5}. 
To our knowledge,   no corresponding deviation result  ($T= \R$ or $T= I$, $I$ some fixed interval depending on $x$) of type \eqref{db.type} is known.
   %he study  of   related supremum properties  of almost periodic Gaussian or Rademacher polynomials is made 
     %the related supremum properties  of almost periodic Gaussian or Rademacher polynomials. See also Queffelec's works  \cite{Q1}, \cite{Q2}. 
   %     Both  questions  have  % -\,a very Gaussian question\,- 
  %    clear   degree   of complexity and limited fields of investigation. 
 From the point of view of the   general development of the   theory of Rademacher or Gaussian processes,     it is   a quite attracting  question, but not only. %as we shall see now. 
 %\vskip 2 pt
 Consider a remarkable example of Gaussian process (here $p$ denotes a prime number)
   $$X_x(t)=\sum_{p\le x } g_p \frac{\cos (\log p)t}{p^{1/2+1/\log x}},
  $$  
   which  has close connection   with   Wintner's random M\"obius function \cite{Wi}, and was   studied by Harper \cite{H}, and later developed in  subsequent papers. Harper building on Hal\'asz \cite{Ha} proved in \cite{H},\, Corollary 2,
  \begin{eqnarray*}\label{harper}
\qq \P\Big\{ \sup_{1\le t\le 2(\log \log x)^2}\sum_{p\le x} g_p \frac{\cos (\log p)t}{p^{1/2+1/\log x}}\ \le \ \Theta (x)\Big\}&=& \mathcal O\big((\log\log\log x)^{-1/2}\big) ,\end{eqnarray*}
as $x$ tends to infinity, where 
 $\Theta(x)\sim\log\log x $.
   On this example the correlation of the normalized process 
  \begin{eqnarray*}  \E\bigg( \frac{X_x (s)}{\sqrt{\E X_x (s)^2 }}\cdot \frac{X_x  (t) }{\sqrt{\E X_x  (t)^2 }}\bigg)& =&  \frac{\sum_{p\le x} a_p^2 \cos(s\log p) \cos(t\log p) }{\sqrt{\displaystyle{\sum_{p\le x} a_p^2   \cos^2(s\log p)   }}\sqrt{\displaystyle{ \sum_{p\le x} a_p^2   \cos^2(t\log p)    }}}.
\end{eqnarray*} 
   can be estimated by the using prime number theorem.  
   %,  for instance  formula (2.26) in    Rosser  and   Schoenfeld \cite{RS}. 
% (the correlation calculation of the whole process, although unnecessary here,   seems less easy to manipulate). %But after remains a substancial   work to do and, although being familiar with Gaussian  processes and analytic number theory,      having also spent   time to study it, we find  the   proof   
 %not everywhere \lq accessible\rq,
 % lacking more mathematical details. 
%We   don't know how to extrapolate from Harper's     work a somewhat generic reducing step  for investigating the previous question, although there should be a way. Hal\'asz's work \cite{Ha} is somehow   richer in mathematical details,  but less in text explanations, and  difficult (see discussion in Harper \cite[p.\,612]{H}). 
\vskip 2 pt %To our knowledge, 
Except for the inspiring work of Harper,  there seems to be   no other attempt in the literature. In the recent work \cite{W1a}, we   deduced from the results obtained however, similar type estimates of the supremum of this  process over general sets. 
Turning now to the  goal of this work, it can  be described   as follows. As a direct approach can only be specific,  we  are interested in  the question of knowing  whether  or not, a general criterion 
%reduction way
 allowing to approximate the deviation of an almost periodic Gaussian polynomial by the one of a periodic one, this in a precise sense, can be established.  % can be performed or not. 
  This is justified by Theorem \ref{semi.asymp.bound.cor} and Corollary \ref{semi.asymp.bound.a}.  By investigating it   we   could eventually show,   by  introducting    a diophantine approximation device, that such an approach can   indeed be worked out, and also    be effective (the error term is sharp). This %has the advantage to 
  shifts the initial problem  to the one of   estimating the smallest approximating periodicity   cycle, and the way it depends on the parameters of the process, which may be easier. Of course the latter problem cannot be solved  generally.
Approximation    by cyclic Gaussian processes appeared already  under another form for  stationary Gaussian processes,  and for a different problem, see Lemma \ref{ap.lm}.  
\subsection{An approximation criterion.}
 Let  ${\mathcal X} =\{ X_{ x}(t) , x\ge 1, t\in\R \}$ be the   class of  $x$-parametrized  almost periodic Gaussian polynomials     defined by
   \begin{eqnarray}\label{gds0} X_x(t) &=& \sum_{1\le k\le x} a_k \big( g_k \cos L_kt+ g'_k\sin L_k t\big)
   %, \qq \qq t\in \R ,\ x\ge 1
  ,
  \end{eqnarray}
 where $\{L_k, k\ge 1\}$ is any increasing unbounded sequence of  positive reals,  $(a_k)_{k\ge 1}$ being any sequence of  real numbers such that 
      \begin{equation}\label{A1}A_x:= \sum_{k\le x} a_k^2\, \uparrow \,\infty \qq \hbox{with $x$.} 
  \end{equation}
  % satisfying assumption \eqref{A1}.  
 To ${\mathcal X}$ we associate  the family $\dot{\mathcal X} =\{ X_{y,x}(u) , 1\le y\le x \}$ of $(y,x)$-parametrized      almost periodic Gaussian polynomials defined  by\begin{eqnarray}\label{gds} X_{y,x}(u)  =  \sum_{y\le k\le x} a_k \big( g_k \cos L_ku+ g'_k\sin L_k u\big),\qq   u\in\R.
 \end{eqnarray}  
%We ask   when ${\mathcal X}$ can be   qualitatively approximated by cyclic Gaussian processes. 
  
  Concretely we consider the problem of bounding   the deviation $$\P\Big\{   \sup_{1\le u\le U} X_{y,x}(u)\le \Theta_{y,x}    \Big\} \qq  \hbox{by  the one of}\qq \P\Big\{   \sup_{1\le u\le U} X^\perp_{ y,x}(u) )\le \Theta_{y,x}    \Big\} ,$$ % We assume in consistency with  the previous sections  that   $A_x = \sum_{k\le x} a_k^2\, \uparrow \,\infty$ with $x$, which is assumption \eqref{A1}.
% \hbox{by  the one of}
 % $$\P\Big\{   \sup_{1\le u\le U} X^\perp_{ y,x}(u) )\le \Theta_{y,x}    \Big\} ,$$ 
with a sharp error term,   where  $X^\perp_{ y,x}$ is a nearby  (cyclic) stationary Gaussian process   
  \begin{equation}X^\perp_{ y,x}(u) = \sum_{   y\le  k\le x} a_k  \big( g_k \cos ( \ell( k)    u) \big) +  g'_k\sin  ( \ell( k)   u)  \big),\qq  u\in\R,
 \end{equation}
and $\ell(k) \in \Q$. 
%This allows one to transfer the initial problem to the one of the deviation $X^\perp_{ y,x}$.  %, only specifically. 
%We introduce a   modulable diophantine approximation.
\vskip 8 pt  {\tenrm A modulable diophantine approximation.} 
   %, which turns up to be    sharp enough in many cases, and modulable. 
  Let $\{N_k, k\ge 1\}$ be a non-decreasing unbounded  test  sequence of  positive  integers. 
Put 
\begin{equation}\label{ellk}\ell(k)=\ell(N_k,k)=\frac{1}{N_k}\big\lfloor N_kL_k\big\rfloor,\qq\quad k\ge 1  .
\end{equation}
%Recall that $L_k=\sum_{j=1}^k l_j^{-1}$. 
Since $\big|\big\lfloor N_kL_k\big\rfloor -N_kL_k\big|\le 1$, we have
  \begin{eqnarray}\label{approxLk}
 \big| \ell( k) -L_k\big|\le \frac{1}{N_k}, \qq\quad k\ge 1   .
 \end{eqnarray}
 
    If $N_k\equiv N$, $N\ge 1$ we have a standard approximation with error term $\le 1/N$, but nothing can be extracted from it for our approach.
     In the following Theorem we find tractable conditions under which this transfer is possible   and   show   that the error term has an exponential decay.
Put  for any interval $I$,
\beq\label{kappU} \k (I)=  \#\{\k:  [N_{\k-1}, N_\k[\subset I\}.
\eeq%  The test sequence $\{N_k, k\ge 1\}$ needs not necessarily grows fast  with $k$.  
 %Consider  the nearby  (cyclic) stationary Gaussian processes  
%  \begin{equation}X^\perp_{ y,x}(u) = \sum_{   y\le  k\le x} a_k  \big( g_k \cos ( \ell( k)    u) \big) +& g'_k\sin  ( \ell( k)   u)  \big),\qq  u\in\R.
% \end{equation} 
%The setting considered is of course too large.
\begin{theorem}[\bf Approximation criterion]\label{sup.kappa.th}  Assume that
\beq \label{Nk.k} N_k\ge k, \qq \quad k\ge 1.
\eeq    

Then  for any reals   $\Theta_{y,x}>0$, $1\le y\le x$, $U\ge 1$, $0<h<  \Theta_{y,x}$,
\begin{equation*}  \P\Big\{   \sup_{1\le u\le U} X_{y,x}(u)\le \Theta_{y,x} -h  \Big\}   \,  \le \, 
\P\Big\{ \sup_{1\le u\le U} X^\perp_{y,x}(u) \le \Theta_{y,x}   \Big\} +2\, \exp\Big\{\frac{- C\,h^2 }{     \D^2   \log \k([1,U]) } \Big\},
  \end{equation*}
where
%  $\D_1=\D$, if $1\le y\le U$ and $\D_1=\D'$, if $1\le U\le y$.      
\begin{equation*}  \D\,=\,  \begin{cases}\ \  y  \   \big( \displaystyle{\sum_{   y\le k\le x} \frac{1 }{N^2_k  }\big)^{1/2}\big( \sum_{   y\le k\le x}a_k^2\big)^{1/2} +
             \sum_{ y\le  k<  \k\atop N_\k \le U}   |a_k|}
  \cr  \    \qquad   +   \displaystyle{\sup_{\k:\,y\le N_\k\le U }\ N_\k  \big( \sum_{   \k\le k\le x} \frac{  1}{N_k^{2 } }      \big)^{1/2}\,\Big(\sum_{   \k\le k\le x}   |a_k|^2    \big)^{1/2}} \quad & \hbox{if $1\le y\le U$,}
  \cr \  \ U \Big( \sum_{   y\le k\le x}\frac{1 }{N^2_k  }\Big)^{1/2}\Big( \sum_{   y\le k\le x}a_k^2\Big)^{1/2}
  \quad & \hbox{if $1\le U\le y$.}
\end{cases} 
 \end{equation*}
 \end{theorem} 
%Theorem \ref{sup.kappa.th} is consistent with the approach made in    Section \ref{s1}, which yields, in view of  Theorem \ref{semi.asymp.bound.cor}, that the study  of   the deviation of $X^\perp_{ y,x}$  may reduce  to the   determination of the lowest approximation periodicity cycle,  tightly  depending   on both $(N_k)_k$ and $(L_k)_k$. 

%The study of the deviation of  $\sup_{1\le u\le U} X^\perp_{y,x}(u)$ is naturally a  specific problem. %  However for $U$ commensurable with $x$ the results of Section \ref{s1} apply. 
\bigskip \par
  {\bf Discussion.}
 %Application: The case $\boldsymbol{ y=1}$. 
 Taking $y=1$ gives,
\begin{equation*}  \P\Big\{   \sup_{1\le u\le U} X_{ x}(u)\le \Theta_{ x} -h  \Big\}   \,  \le \, 
\P\Big\{ \sup_{1\le u\le U} X^\perp_{ x}(u) \le \Theta_{ x}   \Big\} +2\, \exp\Big\{\frac{- C\,h^2 }{     \D ^2   \log \k([1,U]) } \Big\},
  \end{equation*}
where  
 \begin{equation*}  \D  \,=      \, \Big( \sum_{   1\le k\le x}\frac{1 }{N^2_k  }\Big)^{1/2}A_x^{1/2} +
             \sum_{ 1\le  k<  \k\atop \k:N_\k \le U}   |a_k|
     +   \sup_{\k:\,1\le N_\k\le U }\,  \ N_\k  \Big( \sum_{   \k\le k\le x} \frac{  1}{N_k^{2 } }      \Big)^{1/2}\,A_{\k,x}^{1/2}
\,.
 \end{equation*}

--- If for instance $N_k=2^k$, by Remark \ref{Nk.exp}, this simplifies and we can take
\begin{equation*}  \D  =   \,  C\Big\{A_x^{1/2} +
             \sum_{ 1\le  k<  \frac{\log U}{\log 2}}   |a_k|
   %  +   \sup_{\k\le \frac{\log U}{\log 2}} \Big(\sum_{   \k\le k\le x}   |a_k|^2    \Big)^{1/2}
\Big\}.
 \end{equation*} 
\vskip 2 pt 
--- If moreover $|a_k|\le k^{-1/2}$, we can take
 \begin{equation*}  \D  =   \,  C\max\big\{A_x^{1/2} ,
           \sqrt{\log U}\big\}.
 \end{equation*}
 Choose $\Theta = 2HA_x^{1/2} $, $H> 0$, $h=HA_x^{1/2}$. Then
 we obtain \begin{align}\label{case.y=1} &\P\Big\{    \sup_{1\le u\le U}  X_{ x}(u)\le 2HA_x^{1/2}   \Big\}   \cr &  \le \, 
\P\Big\{ \sup_{1\le u\le U} X^\perp_{ x}(u) \le HA_x^{1/2}   \Big\} +2\, \exp\Big\{-\frac{  C\,H^2A_x }{    \max\big\{A_x ,
           {\log U}\big\}   \log \k([1,U]) } \Big\}.
  \end{align}
Assuming $x$ large and taking $U$ be such that  $U\le e^{A_x}$, next $H=(\log \k([1,U]))^{1/2(1 +\eta)}$, $\eta>0$ gives
% $\k([1,U])= \#\{\k:  [N_{\k-1}, N_\k[\subset [1,U]\}$
\begin{align}\label{case.y=1.} &\P\Big\{    \sup_{1\le u\le U}  X_{ x}(u)\le  2A_x^{1/2} (\log \k([1,U]))^{1/2(1 +\eta)}  \Big\}   \cr &  \le \, 
\P\Big\{ \sup_{1\le u\le U} X^\perp_{ x}(u) \le  A_x^{1/2}(\log \k([1,U]))^{1/2(1 +\eta)}   \Big\} +2\, e^{-   C\, \log^\eta \k ([1,U])     }.
  \end{align}\vskip 2 pt
  --- If $ a_p = p^{-1/2}$, $p$ prime, then $A(x)= C\, {\log\log x} $ and  we can take
 \begin{equation*}  \D =   \,  C\max\big\{\sqrt{\log\log x} ,
           \sqrt{\log \k([1,U])}\big\}.
 \end{equation*}
 
 Choose again $\Theta = 2HA_x^{1/2} $, $H> 0$ . We obtain  \begin{align}\label{case.y=1a} \P&\Big\{   \sup_{1\le u\le U}  X_{ x}(u)\le  2H\sqrt{\log\log x}     \Big\}   \cr  &  \le \, 
\P\Big\{ \sup_{1\le u\le U} X^\perp_{ x}(u) \le H\sqrt{\log\log x}  \Big\} +2\, \exp\Big\{-\frac{  C\,H^2\, {\log\log x}}{    \max\big\{ {\log\log x} ,
           {\log U}\big\}   \log \k([1,U]) } \Big\}.
  \end{align}         
     
  For $x$ large and  $  U\le \log x $, $H= (\log \k([1,U]))^{1/2(1 +\eta)}$ this simplifies to 
 \begin{align}\label{case.y=1b} \P&\Big\{   \sup_{1\le u\le U}  X_{ x}(u)\le  \sqrt{\log\log x} (\log \k([1,U]))^{1/2(1 +\eta)}  \Big\}   \cr  &  \le \, 
\P\Big\{ \sup_{1\le u\le U} X^\perp_{ x}(u) \le H\sqrt{\log\log x} (\log \k([1,U]))^{1/2(1 +\eta)} \Big\} +2\, e^{-   C\, \log^\eta \k ([1,U])     }. 
  \end{align}    
  \vskip 2 pt
  The closer investigation of  some   specific cases is made elsewhere.
 
  \bigskip \par
   
  %   \begin{remark}[Harper's example]Recall  that $\log \nu= \sum_{m=1}^\nu \frac1m -\g +\mathcal O(\frac1\nu)$, where $\g$ is Euler's constant.  Let 
% $H_\nu= \sum_{m=1}^\nu \frac1m$, $\nu\ge1$, be the $\nu$-th harmonic number. Thus % from \eqref{log.dev} that  
% $\log (\nu\,e^\g) = H_\nu  +\mathcal O(\frac1\nu)$. 
%These numbers are not integers for $\nu>1$. 
 %But there is an  exact formula  \begin{equation} H_\nu=\frac{1}{\nu!}\, \Big[\begin{matrix}\nu+2\cr2\end{matrix} \Big]  ,
% \end{equation}
% where $\big[\begin{matrix}\nu+2\cr2\end{matrix} \big]$ is a Stirling number of the first kind, which is  integer. We however  are not knowing % don't %Since the harmonic numbers have a logarithmic growth, it would be interesting to 
 %know
  %much
 % about  the  arithmetical structure of these   numbers for $\nu$ prime.   \end{remark}
   
%\vskip 8 pt \begin{remarks}  \label{rem.ap}
 % ${}$\hfill 
%  \vskip 8 pt\begin{itemize}
% \leftskip1mm
% \itemsep=10pt
% \item[{\rm (1)}\ ] 
%Theorem \ref{sup.kappa.th} is consistent with the approach made in    Section \ref{s1}, which yields, in view of  Theorem \ref{semi.asymp.bound.cor}, that the study  of   the deviation of $X^\perp_{ y,x}$  may reduce  to the   determination of the lowest approximation periodicity cycle,  tightly  depending   on both 
%$(N_k)_k$ and $(L_k)_k$. 
 %\item[{\rm (2)}\ ]

   % moderate  deviations of  suprema of  almost periodic Gaussian polynomials  is studied elsewhere. 
 % \vskip 3pt 

%The investigation of     specific cases is made elsewhere.  
 \vskip 3pt  Finally the behavior along lattices of almost periodic Gaussian 
polynomials with linearly independent frequencies is studied  for general non-vanishing coefficient sequences in Section \ref{s6}.   

% \vskip 3pt \vskip 3pt
%In Section \ref{s8}, we gather  other tools which may be of help in this study. 
%We have included in Appendix an overview on Wintner random M\"obius function; we believe these questions should attract  interest   of specialists of stochastic processes, in particular of Gaussian processes.

% \bigskip \par
%The paper is organized as follows.%We next study the   moderate deviations problem for the graduate class of faded Ornstein-Uhlenbeck processes $ {W(T_j)}/{\sqrt{T_j}}$, $1\le j\le n$, $n\to \infty$, where $\  \{T_j,j\ge 1\}$  is an arbitrary increasing   and unbounded sequence of positive reals. This class of processes  appears naturally  in the context of the Skorokhod embedding scheme.

%\vskip 3 pt We finally study moderate  deviations of   almost periodic Gaussian polynomials   \begin{eqnarray}\label{gds} X_x(t) &=& \sum_{k\le x} a_k \big( g_k \cos\l_kt+ g'_k\sin \l_k t\big), \qq \qq t\in \R ,
%\end{eqnarray}
% where $(\l_k)$ is an increasing unbounded sequence of   irrational numbers.  The case   $\l_k= \log p_k$, $p_k$ being the $k$-th consecutive prime, has close connection with   Wintner's random M\"obius function \cite{Wi}. 

  \vskip 5pt    
   \noi  {\it Notation.}   We agree that $\sup_{\emptyset}=0$, $\sum_{\emptyset}=0$.  Let    $I_n$  be the $n\times n$ identity matrix and let $\underline b=(b_1, \ldots ,b_n)\in \R^n$. Let $  I(\underline b)$   denote the $n\times n$ diagonal matrix whose values on the diagonal are the corresponding values of $\underline b$. When $b_i\neq 0$ for each $i=1, \ldots, n$, we   use the notation $\underline b^{\a}=(b_1^{\a}, \ldots ,b_n^{\a})$, $\a$ real.

%%%%%%%%%%%%%%%%%%%%%%%%%%%%%%%%%%%%%%%%%%%%%%%%%%%%%%%%%%%%%%%%%%%%%%%%%%%%%%%%%%%%%%%%%%%%%%

\section{\bf Decoupling inequalities of Gaussian processes.}\label{s2}
These questions originate from the study of Brownian motion by Nelson \cite{N}, Guerra, Rosen and Simon  \cite{GRS} ,   Be\'ska   and  Ciesielski \cite{BC}, Gebelein \cite{G} and Veraar \cite{V}, among other contributors. We    briefly trace back some   results.
 Let $W= \{W(u), u\ge 0\}$ denotes the Brownian motion issued from 0 at time $u=0$.    The    Ornstein-Uhlenbeck process $U= \{U(u), u\ge 0\}$, mostly known example of stationary Gaussian  process, is  defined by
$$  U(t)= W(e^t)e^{-t/2}, \qq t\in
\R. $$
   
As $\E U(s)U(t)= e^{-|s-t|/2}$, the quantity 
$$\upsilon  \, =\, \sum_{n\in \Z} \frac{|\E U(0)U(n) |}{ \E U^2(0)}, $$
is finite and equal to $(\sqrt e-1)^{-1}(\sqrt e+1)  $. It follows from     Klein, Landau and Shucker \cite[Th.\,1]{KLS},  see Theorem \ref{th1.KLS},  that  for any   finite collection $\{f_j, j\in J\}$ of  complex-valued
Borel-measurable functions  of a real variable,  we have the following compact decoupling inequality 
\begin{eqnarray}\label{dec.U}
\Big|\E  \prod_{j\in J}f_j\big(U(j)\big) \Big|\ \le\  \prod_{j\in J}\big\|f_j\big(U(0)\big) \big \|_{\upsilon}.\end{eqnarray}
 %\begin{theorem}\label{th1.KLS}  Assume that 
%\begin{eqnarray}\label{hyp.dec} p(\X) \ =\ \sum_{n\in \Z} \frac{|\g(n) |}{ \g(0)}<\infty .
%\end{eqnarray}
%Then for any finite collection $\{f_j, j\in J\}$ of  complex-valued
%Borel-measurable functions  of a real variable,
%\begin{eqnarray}\label{dec}
%\Big|\E  \prod_{j\in J}f_j\big(X_j\big) \Big|\ \le\  \prod_{j\in J}\big\|f_j\big(X_0\big) \big \|_{p(\X)}.\end{eqnarray}
 %\end{theorem}
\vskip 2 pt  
Guerra, Rosen and Simon \cite[Lemma III.11]{GRS} 
% is Nelson's hyper-contractive  estimate, which  can be reformulated as follows
%\begin{equation}\label{nelson} |\E f(U)h(V) |\le   \|f\|_p\|h\|_q,
%\end{equation} 
%%where $(p-1)(q-1)\ge  \rho^2$. One can take in particular $p=q=1+|\rho|$.
%We have given Guerra, Rosen and Simon formulation of Nelson's estimate \cite{GRS}, which was originally stated for the Ornstein-Uhlenbeck
%process. 
 earlier proved % for this process
  that
 \begin{equation}\label{grs} \Big|\E \prod_{j=1}^n f_j(U(ja))\Big |\le   \prod_{j=1}^n \|f_j(U(0))\|_p ,
\end{equation}
for all integers $n$, where $a>0$ and $p=(1-e^{-na})^{-1}(1+e^{-na})$.      Much attention has been paid to this process concerning  this sort of questions. A first   relevant and little known    correlation estimate is   Gebelein's  inequality \cite{G}.   Let $\nu$ be the centered
normalized Gauss measure on $\R$. Let $(U,V)$ be a Gaussian pair with $U\buildrel{\mathcal D}\over {=}V\buildrel{\mathcal D}\over {=}\nu$
and let $\rho= \E U V$. Then for any $f,h\in L^2(\nu)$ with $\E f(U)=\E h(V)=0 $,
\begin{equation}\label{Gebel} |\E f(U)h(V) |\le |\rho| \|f\|_2\|h\|_2.
\end{equation}  
 See also  Be\'ska   and  Ciesielski \cite{BC} for some of its consequences (Borel-Cantelli Lemma, iterated log law, Levy's norm \ldots), and Veraar \cite{V} for extensions and applications to standard limit theorems. 
  An analog result to \eqref{Gebel} is Nelson's hyper-contractive  estimate, which  can be reformulated as follows
\begin{equation}\label{nelson} |\E f(U)h(V) |\le   \|f\|_p\|h\|_q,
\end{equation} 
where $(p-1)(q-1)\ge  \rho^2$. One can take in particular $p=q=1+|\rho|$.
This is Guerra, Rosen and Simon formulation of Nelson's estimate \cite{GRS},   originally stated for the Ornstein-Uhlenbeck
process. 

    Given a   centered Gaussian vector  $X=\{X_i, 1\le i\le n\}$, we say that a $X$ satisfies a decoupling inequality when for some real   $p  $     greater than 1, and some explicit constant $\mathcal Q(X,p) $,  the following inequality  %find if possible optimal conditions under which      
    \beq\label{dec.ineq.intro}
 \E\Big( \prod_{i=1}^n f_i(X_i)\Big)
    \le    \mathcal Q(X,p)   \prod_{i=1}^n  \big\| f_i(X_i) \big\|_p, 
\eeq 
 holds for any complex-valued measurable  functions $f_1, \ldots, f_n$ such that $f_i\in L^p(\R)$, for all $1\le i\le n$. This exhibits a remarkable structural independence property, and this is in a same time a nice problem in the theory. 
 
 A $p$-region $S \subset ]1, \infty)$ is said to be   admissible, given $X$,  if \eqref{dec.ineq.intro}
 holds for all $p\in S$. The search of optimal  admissible regions  is an interesting question. A nearly optimal  disconnected  region is recently found in  \cite{W1}, using matrix method.
 %based on  a   simultaneous diagonalization argument.
   \vskip 1 pt 
 
%We have taken into account of this in the writing of \cite{W5}.
\vskip 4 pt The proof of Theorem \ref{semi.asymp.bound.cor} uses a      decoupling inequality for cyclic stationary Gaussian processes
 due to  Klein, Landau and Shucker.  
\begin{theorem}[\cite{KLS},\,Th.\,3]\label{th3.KLS} Let $m$ be a positive integer. Let $\X=\{X_n, n\in\Z/m\Z\}$ be a cyclic stationary Gaussian process, and let  $\{f_n, n\in\Z/m\Z\}$ be a collection of  complex-valued
Borel-measurable functions of a real variable. Then
\begin{eqnarray}\label{dec.cyclic}
\Big|\E  \prod_{n\in\Z/m\Z}f_n\big(X_n\big) \Big|\ \le\  \prod_{n\in\Z/m\Z}\big\|f_n\big(X_0\big) \big \|_{p(\X)},\end{eqnarray}
 where 
\begin{eqnarray}\label{hyp.dec.cyclic} p(\X) \ =\ \sum_{n\in\Z/m\Z} \frac{|\E X_0X_n |}{ \E X_0^2}.
\end{eqnarray}
\end{theorem}
The coefficient $p(\X)$ is called the decoupling coefficient of the Gaussian process $\X$.
The stationarity is understood
 in the sense of the additive group structure of  $\Z/m\Z$.

    Now  let $\X=\{X_j, j\in \Z\}$ be a centered Gaussian stationary  sequence,  $\g(n) = \E X_0X_n$, $n \in \Z$. 
A remarkable application of Theorem \ref{th3.KLS} is  \begin{theorem}[\cite{KLS},\,Th.\,1]\label{th1.KLS}  Assume that 
\begin{eqnarray}\label{hyp.dec} p(\X) \ =\ \sum_{n\in \Z} \frac{|\g(n) |}{ \g(0)}<\infty .
\end{eqnarray}
Then for any finite collection $\{f_j, j\in J\}$ of  complex-valued
Borel-measurable functions  of a real variable,
\begin{eqnarray}\label{dec}
\Big|\E  \prod_{j\in J}f_j\big(X_j\big) \Big|\ \le\  \prod_{j\in J}\big\|f_j\big(X_0\big) \big \|_{p(\X)}.\end{eqnarray}
 \end{theorem}
Theorem \ref{th1.KLS} follows from Theorem \ref{th3.KLS} by approximating   stationary Gaussian processes by cyclic Gaussian processes with arbitrary large periods, which assumption \eqref{hyp.dec} renders possible. 
More precisely (\cite[p.\,707]{KLS}), 
%let $r(n)$ be a function of positive type on $\Z^d$ such that 
% $ {\sum_{n\in \Z^d} |r(n)|<\infty}.$  Then for all positive integers $N$, 
% $$  r_N(n)\,=\, \sum_{k\in \Z^d} r(n+Nk) 
% $$
%is a well defined function of positive type on $\Z^d$, periodic in each variable $n_i$, $i=1,\ldots, d$, with period $N$, and such that 
%$$ \lim_{N\to \infty} r_N(n)\,=\, r(n)$$
%for all $n\in \Z^d$.    
 %\begin{problem}\label{ap.ca}
   \begin{lemma}[Approximation by large periods]\label{ap.lm} Let $r(n)$ be a function of positive type on $\Z$ such that 
 $$\sum_{n\in \Z} |r(n)|<\infty.$$ Then for all positive integers $N$, 
 $$ r_N(n)\,=\, \sum_{k\in \Z} r(n+Nk)
 $$
 is a well defined function of positive type on $\Z$, periodic in each variable $n_i$, $i=1,\ldots, d$, with period $N$, and such that 
 $$ \lim_{N\to \infty} r_N(n)\,=\, r(n)$$
 for all $n\in \Z$. \end{lemma}
By definition, this  argument only applies    in the stationary case.
%  This  remarkable  inequality, which 
  %so nicely condenses the independence properties of these Gaussian sequences,  is Theorem 1
  %Theorem 3 ($d=1$)
  % in Klein, Landau and Shucker \cite{KLS}.
   %, see Theorem \ref{th3.KLS}. 
 It is natural to ask    concerning the problem studied, what can be drawn  when a stationary Gaussian process can be   qualitatively approximated by cyclic Gaussian processes.
%We only address this question in this project, and will develop it more elsewhere.   
\vskip 3 pt   \vskip 2 pt

 In   
asymptotic analysis  of Gaussian sequences and associated partial maxima, it does not seem one can develop   such   type of inequality, or maybe   in the  setting of Gaussian vectors with   entry-wise positive covariance matrix, see for instance  Johnson and Tarazaga \cite{JT}.
  Typical conditions are: either 
\begin{equation}\label{atgp} \big|\E X_iX_j\big|\le \rho(|i-j|)
\quad \hbox{where}\quad  \rho(n)= o \big({1}/{\log n}\big) \quad \hbox{as}\  n\to \infty,
\end{equation}
or, \begin{equation}\label{atgp1} \big|\E X_iX_j\big|\le \rho(|i-j|)
\quad \hbox{%for $|i-j|\ge N_0$ 
where}\quad  \sum_{n=1}^\infty \rho^2(n)<\infty.
\end{equation}
 We quote a classical  useful lemma (with many variants)  in the study of asymptotic behavior of Gaussian dependent sequences. We have noted for any function $g:\R^2\to \R$ and any real $h$ the difference  operators
 \begin{eqnarray*}\ \D^1_{h}\ g(u,v)= g(u+h,v)-g(u,v),  \qq\qq\qq\qq
 \cr &\cr\ \ \ \D^2_{h}\ g(u,v)= g(u ,v+h)-g(u,v).\qq\qq\qq\qq\,  \end{eqnarray*}
 \begin{lemma}[\cite{W9},\,Lemma\,1]
Let $X=(X_1,\dots ,X_N)$ be a Gauss\-ian
centered vector such that $\E X_n^2=1$ for $1\le n\le N$ and let
$r(n,m)=\E X_nX_m$ be  its covariance function. Let $A$ be a partition
of $\{1,\dots, N\}$ and denote by $\s$ a generic element of $A$.
Let $x=(x_1,\dots ,x_N)$ and $y =(y_1,\dots ,y_N)$ with distinct
coordinates, be such that $-\infty < x_n< y_n< +\infty $, for $1\le
n\le N$. Denote also by $I_n$ the interval $(x_n,y_n)$, and put for each
$\s\in A$,
$$
V_\s= \prod_{n\in \s} I_n,\qq V= \prod_{\s\in A} V_\s, \qq X_{(\s)}= (X_n, n\in \s).
$$
Then 
$$\Big|\P\{ X\in V\}-\prod_{\s\in A}\P\{ X_{(\s)}\in V_\s\}\Big|\le \frac12\sum_{\s\not=\s'}
\sum_{n\in \s }\sum_{m\in  \s'} k(n,m)|r(n,m)|,
$$
where 
$$ k(n,m)=\int_0^1\D^1_{y_n,x_n}\circ\D^2_{y_n,x_m}\big(\Phi(x_n,x_m,\l r(n,m))\big) \dd \l,$$
and 
$$ \Phi(x ,y,\rho)=\frac{1}{2\pi\sqrt{1-\rho^2}}\exp\big\{-\frac{x^2+y^2-2\rho xy}{2(1-\rho^2)}\big\}.$$
\end{lemma}
  \begin{definition}\label{dec.coeff.definition} Let $ X=\{X_i, 1\le i\le n\}$ be a centered Gaussian vector with non-degenerated components. The decoupling coefficient $p(X)$ of $X$ is defined by
 \begin{eqnarray}\label{dec.coeff.def} 
 p(X)&=& \max_{i=1}^n\sum_{1\le j\le n} \frac{|\E X_iX_j|}{\E X_i^2}
\,.
\end{eqnarray}
\end{definition}
When $X$ is stationary, $ \E X_iX_j =\g(|i-j|)$, \begin{eqnarray*}  
p(X)&=&\max_{i=1}^n\sum_{1\le j\le n} \frac{|\g(i-j)|}{\g(0)},\end{eqnarray*}
and so 
\begin{eqnarray}\label{dec.coeff.def.stat} 
\sum_{0\le h\le n-1} \frac{|\g(h)|}{\g(0)}\le p(X)\le 2\sum_{0\le h\le n-1} \frac{|\g(h)|}{\g(0)}.\end{eqnarray}
  
%The theorem below is of relevance in \cite{W6} where $W= \{W(u), u\ge 0\}$ being Brownian motion issued from 0 at time $u=0$. This process    arises naturally in    questions invlving well-known   Skorokhod embedding scheme. 
 \begin{theorem}[\cite{W3},\,Th.\,2.3]\label{dec.ineq} Let $ X=\{X_i, 1\le i\le n\}$ be a centered Gaussian vector  with invertible  covariance matrix $C$, and let  $\s_i^2=\E X_i^2>0$ for each  $1\le i\le n$.  
Let $\b\ge 1$ be chosen so that $ \bar{\b}:=
\frac{(\max \s_i^2)}{(\min \s_i^2)}\vee\b>1$, and let $p$ be such that 
\begin{eqnarray}\label{cond}
p\, \ge \,   \bar{\b}\, p(X).
\end{eqnarray}
 Then for any complex-valued measurable  functions $f_1, \ldots, f_n$ such that $f_i\in L^p(\R)$, for all $1\le i\le n$, the following inequality holds true, 
\begin{equation}\label{ineq}
\bigg|\,\E\Big( \prod_{i=1}^n f_i(X_i)\Big)\,\bigg|
\,\le \,\frac{ \big(\prod_{i=1}^n\s_i\big)^{\frac{1}{p}}}{\big(  1-  {1}/{\bar{\b}} \big)^{\frac{n}{2}(1-\frac{1}{p})}\det(C)^{\frac{1}{2p}}}\
  \prod_{i=1}^n  \big\| f_i(X_i) \big\|_p.
\end{equation}
 \end{theorem}
The interesting case is $1\ll p(X)\ll n$. We refer to   \cite{W3}, Sections 4-6,  for some    remarkable  examples. 
\vskip 2pt
  Naturally  Theorems \ref{Ac1c}, \ref{gen.case.th} have by means of Theorem \ref{dec.ineq} a decoupling version counterpart. 
 %%%%%%%%%%%%%%%%%%%%%%%%%%%%%%%%%%%%%%%%%%%%%%%%%%%%%%%%%%%%%%%%%%%%%%%%%%%%%%%%%%%%%%%%%%%%%%

 \section{\bf Proof of Theorem \ref{semi.asymp.bound.cor}}\label{s3}    %  

 \subsection{Intermediate results}
 
We first establish some results of proper interest, and discuss several points. 
In view of \eqref{dec.coeff.def.stat} the decoupling coefficient  to consider   in Theorem \ref{semi.asymp.bound.cor}  is 
\begin{eqnarray}\label{p(n).trigo}  p_{y,x}(n)\,=\,
 \frac{1}{A(y,x) }\ \sum_{j=0}^{n-1} \Big|\sum_{ y\le k\le x } a_k^2\cos  2\pi  j_k\frac{j}{n} \Big|.
  \end{eqnarray}
 It is natural in view of the theory of  Riemann sums \cite[Ch.\,XI]{We} to link 
   both quantities below 
   \begin{eqnarray*}
   %\label{int.trigo}
   \frac1n\sum_{j=0}^{n-1} \Big|\sum_{ y\le k\le x } a_k^2\cos  2\pi  j_k\frac{j}{n} \Big|, \qq\qq \int_0^1 \Big|\sum_{y\le k\le x} a_k^2\cos 2\pi  j_ku  \Big|\dd u.
  \end{eqnarray*}

If in place of the $L^1$-norm we have a $L^2$-norm, then  by the well-known formula of \lq mechanical quadrature\rq
   \begin{equation}\label{mec.quad}  
   \frac{1}{2\pi}\int_{-\pi}^{\pi}P(x)\dd x = \frac{1}{2N}\sum_{\nu=-N+1}^{N}P\Big(\frac{\nu \pi}{N}\Big)     ,
  \end{equation}
  valid for any trigonometric polynomial $P(x)$ whose degree does not exceed $2N-1$, it follows by squaring out that both quantities coincide. See Grenander  and Szeg\"o \cite[p.\,68]{GS}. The   problem of finding general estimates of   $p_{y,x}(n)$ and associate integral is  quite hard; it 
   % as well as  the difference between these two characteristics 
   corresponds to  a weighted form  of Littlewood Hypothesis for cosine sums, see    \cite[Remark\,5.2]{W3}. 
 \vskip 2 pt   
    Another remarkable fact is  that if  in place of having a Riemann sum (thus based on periodic points), we had a random Riemann sum, namely, that   the points
  are chosen independently and at random from the intervals $
I^{(n)}_k= [\frac{k-1}{n}, \frac{k}{n}]$, $k=1,\ldots n$, then the problem of comparing the two corresponding characteristics becomes  much simpler.   The   random Riemann sums of $f\in L^2$ converge almost everywhere to   the integral of $f$. Stronger estimates of deviation type and also other almost everywhere type results exist. See  C.\,S. Kahane \cite{Ka},  Pruss \cite{P}. 
% \begin{theorem}[Kahane \cite{Ka}]\label{ka} 
% Assume that $f\in L^p[0,1]$ with $p>2$ and suppose that the points
%$x^{(n)}_h$ are chosen independently and at random from the intervals $
%I^{(n)}_h= [\frac{h-1}{n}, \frac{h}{n}]$, $h=1,\ldots n$. Then,   
%$$\frac{1}{n} \sum_{h=1}^n f(x^{(n)}_h) \quad \longrightarrow \quad \int_0^1 f(x) dx,
%$$
%almost everywhere.\end{theorem}Later Pruss   proved the case $p=2$ and established complete convergence. The   statements obtained \cite[Th.\,1, Cor.\,3]{P} are optimal. T 
 \vskip 3 pt 
    The following proposition is a variant   of Proposition\,5.1 in \cite{W3}.    %needed in Section \ref{6.II}. 
  \begin{proposition}  \label{p(n)a} We have the following estimate,
\begin{eqnarray*} \Big|p_{y,x}(n)-\frac{n }{A(y,x) }\int_0^1 \Big|\sum_{  y\le k\le x } a_k^2\cos  2\pi  j_ku\Big|\dd u\Big|&\le&   \frac{ 2\pi }{A(y,x)}  \ \sum_{y\le k\le x} j_ka_k^2 .
\end{eqnarray*}
Further,
   \begin{eqnarray*}  p_{y,x}(n)&  \le & 
\frac{1}{  A(y,x)     }\Big\{n\Big(\sum_{k\le x} a_k^4\Big)^{1/2}+2\pi   \sum_{ y\le k\le x }  j_k a_k^2\Big\}.
  \end{eqnarray*}  
  \end{proposition}
   
\begin{proof}   Let $f(u)=|\sum_{y\le k\le x} a_k^2\cos 2\pi  j_ku|$ and  $R_nf=\frac{1}{n}\sum_{j=0}^{n-1} f(\frac{j}{n})$ be the Riemann sum  
  and $I(f)= \int_0^1 \big|\sum_{y\le k\le x} a_k^2\cos 2\pi  j_ku  \big|\dd u$.   We  express $p_{y,x}(n)$ as follows: 
\begin{eqnarray}
\label{dec.coeff.period.pol.}
p_{y,x}(n) &=& \frac{n R_n f}{A_x}  \,   \, =\, \frac{n I(f)}{A_x}+ \frac{n (R_n f-I(f))}{A_x}.
\end{eqnarray}  
 Let $\p(u)= \sum_{y\le k\le x} a_k^2\cos 2\pi  j_ku$, $f(u)=|\p(u)|$. 
    Let $R_nf=\frac{1}{n}\sum_{j=0}^{n-1} f(\frac{j}{n})$ be  the Riemann sum of order $n$ of $f$. 
%From \eqref{p(n).trigo},
%\begin{eqnarray}\label{p(n)1} p(n)&  = & \frac{nR_nf}{A_x } , 
%  \end{eqnarray} recalling that $\theta =1/n$. 
Plainly,
\begin{eqnarray}\label{estp}\big|R_nf-\int_0^1 f(u)\dd u\big|&\le&\Big|\sum_{j=0}^{n-1}\int_{\frac{j}{n}}^{\frac{j+1}{n}}\big(f(u) -f(\frac{j}{n})\big) \dd u\Big|
\cr &\le&\sum_{j=0}^{n-1}\int_{\frac{j}{n}}^{\frac{j+1}{n}}\big|f(u) -f(\frac{j}{n})\big| \dd u 
\cr &\le&\sum_{j=0}^{n-1}\int_{\frac{j}{n}}^{\frac{j+1}{n}}\big|\p(u) -\p(\frac{j}{n})\big| \dd u \ \le \ \frac{\|\p'\|_{\infty}}{n} .\end{eqnarray}
  As $\p'(u)=  - 2\pi   \sum_{y\le k\le x} j_ka_k^2\sin 2\pi  j_ku$,  and so $\|\p'\|_{\infty}\le   2\pi  \sum_{k\le x} j_ka_k^2$, we deduce that  \begin{eqnarray}\label{estR}\Big|nR_nf-n\int_0^1 f(u)\dd u\Big|&\le&   2\pi \sum_{k\le x} j_ka_k^2 .\end{eqnarray}

And so, since $p_{y,x}(n)= \frac{nR_nf}{ A(y,x)  }$, 
 \begin{eqnarray}
  \label{estR.a}
 \Big|p_{y,x}(n)-\frac{n }{ A(y,x)  }\int_0^1 \Big|\sum_{y\le k\le x} a_k^2\cos  2\pi  j_ku\Big|\dd u\Big|&\le&   \frac{ 2\pi }{A(y,x) }  \ \sum_{y\le k\le x} j_ka_k^2 .\end{eqnarray}

By Cauchy-Schwarz's inequality, next using orthogonality of the system $\{\cos    2\pi  j_ku, k\le x\}$,
\begin{eqnarray*} \int_0^1 f(u)\dd u=\int_0^1 \Big|\sum_{y\le k\le x} a_k^2\cos  2\pi  j_ku\Big|\dd u &\le & \Big(\int_0^1 \Big|\sum_{y\le k\le x} a_k^2\cos    2\pi  j_ku\Big|^2 \dd u\Big)^{1/2}
\cr &\le &\Big(\sum_{y\le k\le x} a_k^4\Big)^{1/2}
.
\end{eqnarray*}
%Combining now with \eqref{estR}, we  get
%\begin{eqnarray*} nR_nf &\le & n\Big(\sum_{y\le k\le x} a_k^4\Big)^{1/2}+2\pi \sum_{y\le k\le x} j_ka_k^2.
%\end{eqnarray*}
% \vskip 3 pt  
Consequently, by \eqref{estR.a},
 \begin{eqnarray*}
 %\label{p(n)b}
  p_{y,x}(n)&  \le & 
\frac{1}{A(y,x) }\Big\{n\Big(\sum_{y\le k\le x} a_k^4\Big)^{1/2}+ 2\pi \sum_{y\le k\le x} j_ka_k^2\Big\}.
  \end{eqnarray*}  
 \end{proof}

  \bigskip\par
 We  also need  the following  Proposition. \begin{proposition}\label{p(n)b}   Let $0<\e\le 1$,   $n\ge   1/\e$ and let $z=\big\lceil n\e\big\rceil$. 
\vskip 2 pt  {\rm (i)} For all $\Theta\ge 0$, and all $x\ge 1$,
\begin{eqnarray*}
%\label{q1esta}
\P\Big\{ \sup_{m=0,\ldots, z}X(\frac{m}{n})
 \, \le \, \Theta\Big\}
     &\le& e^{- \e\, \P\{X(0)>\Theta\}( n /p_{y,x}(n) )}.
     \end{eqnarray*}
   
   {\rm (ii)}   Take $n$ such that  $n\ge \max( 2\pi \sum_{y\le k\le x} j_ka_k^2, 1/\e)$. Then\begin{equation*}
%\label{q1estb} 
%\P\Big\{ \sup_{0\le t\le 1}\widetilde X_{y,x}(t)\, \le \, \Theta\Big\}\,\le \, 
\P\Big\{ \sup_{m=0,\ldots, z}X(\frac{m}{n})
 \, \le \, \Theta\Big\}
 %\cr  &= & \P\Big\{ \sup_{j=0,\ldots, n-1}X(\frac{j}{n})\ \le \ \Theta\Big\}
\,\le \, \exp\Big\{ \frac{- \, \e\, \P\{X(0)>\Theta\}   A(y,x)}{
\big((\sum_{y\le k\le x} a_k^4)^{1/2}+1\big)  }\Big\}
 .   \end{equation*}
    \end{proposition} 
Claim (i) shows that   the bound obtained is sharp when  $ p_{y,x}(n)\ll n  $. See \cite{W3} for examples.  \begin{proof}The map $m\mapsto\E \dot{X}(\bar{m})\dot{X}(\overline{m+j})=\E X( \frac{m}{n} )X(\frac{m+j}{n} )=\sum_{k\le x} a_k^2\cos2\pi j_k \frac{j}{n}$
%\begin{equation}\label{constant}m\mapsto\E \dot{X}(\bar{m})\dot{X}(\overline{m+j})=\E X( \frac{m}{n} )X(\frac{m+j}{n} )=\sum_{k\le x} a_k^2\cos2\pi k \frac{j}{n}.\end{equation}
being  constant over $\Z/n\Z$, for any $j\in\Z/n\Z$, the process $\dot{X}$ is thus a cyclic stationary Gaussian process on $\Z/n\Z$, with respect to the additive group structure of $\Z/n\Z$. 
 Thus Theorem \ref{th3.KLS} is in force.   Put  
\begin{equation} f_m =\begin{cases} \chi {\{]-\infty, \Theta]\}},  &\quad 1\le m\le z\cr
1&\quad z<m< n.
 \end{cases}\end{equation}

\vskip 4 pt
 
  We prove  (i). Applying Theorem \ref{th3.KLS} to   $f_m\big(\dot{X}(\bar{m})\big),\ m=0, 1, \ldots, n-1$, we obtain
\begin{eqnarray}\label{q1esta}
\P\Big\{ \sup_{m=0,\ldots, z}X(\frac{m}{n})
 \, \le \, \Theta\Big\}
 &=&  \P\Big\{ \sup_{m=0,\ldots, z}\dot{X}(\bar{m})\, \le \, \Theta\Big\}\,\le \,  \P\{\dot{X}(0)<\Theta\}^{z/p_{y,x}(n)}
\cr &= &   \P\{X(0)<\Theta\}^{z/p_{y,x}(n)}
  \ = \ \big(1- \P\{X(0)>\Theta\}\big)^{z/p_{y,x}(n)}
   \cr &= & e^{(z/p_{y,x}(n))\log(1- \P\{X(0)>\Theta\})}
  \ \le\ e^{- \P\{X(0)>\Theta\}{z/p_{y,x}(n)}}
  \cr &=& e^{- \e \P\{X(0)>\Theta\}( n /p_{y,x}(n) )}.\end{eqnarray}
 since $\log x \le x-1$ for all $x>0$.

  % As $(\sum_{y\le k\le x} a_k^4)^{1/4}\le (\sum_{y\le k\le x} a_k^2)^{1/2}$, we have   $n>(\sum_{y\le k\le x} a_k^4)^{1/2}$. 
  \vskip 4 pt
We now prove (ii).     At first, by   Proposition \ref{p(n)a},
 \begin{eqnarray*}
 %\label{p(n)b}
  p_{y,x}(n)\,  \le \, 
\frac{1}{A(y,x) }\Big\{n\Big(\sum_{y\le k\le x} a_k^4\Big)^{1/2}+ 2\pi \sum_{y\le k\le x} j_ka_k^2\Big\} .
  \end{eqnarray*}
Take $n$ such that  $n\ge 2\pi \sum_{y\le k\le x} j_ka_k^2$. Then  
$$ p_{y,x}(n)   \le   \frac{n}{A(y,x) }\big((\sum_{y\le k\le x} a_k^4)^{1/2}+1\big),$$ and so 
% Let $0<\e<1$ and $z=n\e$. We choose   $n$  in fact large so that 
 
%For instance take   $n$ such as
% $$%n \ge  
%  n 
 %=\frac{\psi}{N}\frac{ n}{A(y,x) } =\frac{\psi}{N}p(n)
% \ge      \big(\frac{2\pi}{\e} \big)\sum_{y\le k\le x} j_ka_k^2.$$
%Then 
%$$p_{y,x}(n)\,  \le \, 
%\frac{n\e}{A(y,x) }\Big\{  \Big(\sum_{y\le k\le x} a_k^4\Big)^{1/2}+1  \Big\} .$$
 \begin{equation}\label{a4.2.z}  
  \frac{n }{p_{y,x}(n) }\ge \frac{A(y,x)}{\big((\sum_{y\le k\le x} a_k^4)^{1/2}+1\big)}      .
  \end{equation}

\vskip 5 pt   
Under assumption  \eqref{A1.a.yx}, for some $0<\eta<1$,
$$\Big(\sum_{y\le k\le x} a_k^4\Big)^{1/2}\le \frac{A(y,x)^{1-\eta}}{\sqrt{\log A(y,x)}} ,$$ 
so that  the rigth-term in \eqref{a4.2.z} is large with $A(y,x)$. 
  It follows  from \eqref{a4.2.z} that,
 
%\P\Big\{ \sup_{0\le t\le 1}\widetilde X_{y,x}(t)\, \le \, \Theta\Big\}\,\le \,  \P\Big\{ \sup_{j=0,\ldots, n-1}X\Big(\frac{j}{n} \Big)\,\le \,  \Theta\Big\}
  \begin{equation}\label{q1estb} 
%\P\Big\{ \sup_{0\le t\le 1}\widetilde X_{y,x}(t)\, \le \, \Theta\Big\}\,\le \, 
 \P\Big\{ \sup_{j=0,\ldots, z}X\Big(\frac{j}{n} \Big)\,\le \,  \Theta\Big\} 
%\cr  &= & \P\Big\{ \sup_{j=0,\ldots, n-1}X(\frac{j}{n})\ \le \ \Theta\Big\}
\ \le \ \exp\Big\{ {- \,\frac{\e\, \P\{X(0)>\Theta\}   A(y,x)}{
\big((\sum_{y\le k\le x} a_k^4)^{1/2}+1\big) }}\Big\}
 ,
   \end{equation}
 for all $\Theta\ge 0$, and all $x\ge 1$. 
 \end{proof} 
  \vskip 8 pt 
 \subsection{\bf Proof of  Theorem \ref{semi.asymp.bound.cor}}  Recall  that Mills'  ratio is  $R(x) =
 e^{x^2/2}\int_x^\infty e^{-t^2/2}\ \dd t$, and that  
 $  R(x)\ge \frac1{1+x}$,    $x\ge 0$, \cite[p.\,178]{Mi}. 
  Let 
 $\Theta = H\sqrt{A(y,x)} $, $H> 0$. 
 Then, noting that   $\E X(0)^2=A(y,x)$,
    $$\P\{X(0)>\Theta\}= \P\Big\{\frac{X(0)}{\|X(0)\|_2}>\frac{\Theta}{\sqrt{A(y,x)}}\Big\}=  e^{-H^2/2}\frac{ R(H)}{\sqrt {2\pi}}\ge  \frac{e^{-H^2/2}}{\sqrt {2\pi}( H+1)}.$$
 Let $0<\eta<1$ be defined according to  assumption   \eqref{A1.a.yx}. Choose 
  $ H=\sqrt{2\eta \log A(y,x)}$.  
 Then
 \begin{eqnarray*}
  \P\{X(0)>\Theta\}A(y,x)&\ge &  C\ \frac{e^{-H^2/2} A(y,x)}{H}\ =\ C\,\frac{ A(y,x)^{1-\eta}}{\sqrt{2\eta \log A(y,x)}}.
  \end{eqnarray*}
  
 Let $0<\e\le 1$  and $n$ such that  $n\ge \max( 2\pi \sum_{y\le k\le x} j_ka_k^2, 1/\e)$, $z= \lceil n\e \rceil$.  As $\Theta = \sqrt{2\eta A(y,x)\log A(y,x)}$, it follows from Proposition \ref{p(n)b}-(ii) that 
 \begin{align*}
%  \P\Big\{ \sup_{0\le t\le 1} \widetilde X_{y,x}(t) \le \sqrt{2\eta A(y,x)\log A(y,x)}\Big\}
 \P\Big\{ \sup_{j=0,\ldots, z}X\Big(\frac{j}{n} \Big)\,\le \,  \Theta\Big\} \ \le  \ e^{-\,\frac{C \e\,A(y,x)^{1-\eta}}{  \sqrt{ \eta  ((\sum_{y\le k\le x} a_k^4)^{1/2}+1 )\log A(y,x)}}}.
 \end{align*}

 If $B=\sum_{k\ge 1} a_k^4<\infty$, then
 \begin{align*}
 % \P\Big\{ \sup_{0\le t\le 1}\widetilde X_{y,x}(t) \le\sqrt{2\eta A(y,x)\log A(y,x)}\Big\}
 \P\Big\{ \sup_{j=0,\ldots, z}X\Big(\frac{j}{n} \Big)\,\le \,  \Theta\Big\} \ \le  \ e^{-\,\frac{C  \e\,  A(y,x)^{1-\eta}}{  \sqrt{ \eta (B+1)\log A(y,x)}}}.
 \end{align*}
%Let $\e\ge (2\pi \sum_{y\le k\le x} j_ka_k^2)^{-1}$ and pick $n \ge  2\pi \sum_{y\le k\le x} j_ka_k^2 $.  Thus $\frac1n\le \e$, hence $z\ge 1$. 
As
$$\P\Big\{ \sup_{0\le t\le \e}\widetilde X_{y,x}(t)\, \le \, \Theta\Big\}\,\le \,  \P\Big\{ \sup_{j=0,\ldots, z}X\Big(\frac{j}{n} \Big)\,\le \,  \Theta\Big\},$$
we get 
$$\P\Big\{ \sup_{0\le t\le \e}\widetilde X_{y,x}(t)\, \le \, \sqrt{2\eta A(y,x)\log A(y,x)}\Big\}\,\le \,    e^{-\,\frac{C \e\, A(y,x)^{1-\eta}}{  \sqrt{ \eta (B+1)\log A(y,x)}}}.$$
The case $B=\sum_{k\ge 1} a_k^4<\infty$ follows identically. 

Now let $\eta=1$. Choose 
again $\Theta = H\sqrt{A(y,x)}   $ but with  $ H=\sqrt{2  \log (A(y,x)/V(y,x)) }$.  
Then
%\Theta =H  \sqrt{A(y,x)}=\sqrt{2A(y,x)  \log \frac{A(y,x)}{V} }  
 \begin{eqnarray*}
  \P\{X(0)>\Theta\}A(y,x)&\ge &  C\ \frac{e^{-H^2/2} A(y,x)}{H}\ =\ C\,\frac{V(y,x)}{\sqrt{2 
   \log ( {A(y,x)}/{V(y,x)}) }},  \end{eqnarray*}
  and thus
   \begin{align*}
  \P\Big\{ \sup_{j=0,\ldots, z}X\Big(\frac{j}{n} \Big)\,\le \,  \sqrt{2A(y,x)   \log ( {A(y,x)}/{V(y,x)}) }\Big\} \ \le  \ e^{-\,\frac{C \e\,V(y,x)}{  \sqrt{   ((\sum_{y\le k\le x} a_k^4)^{1/2}+1 ) \log ( {A(y,x)}/{V(y,x)})}}}.
 \end{align*}

We get 
% for  $ (2\pi \sum_{y\le k\le x} j_ka_k^2)^{-1} \le \e \le  1$,
$$\P\Big\{ \sup_{0\le t\le \e}\widetilde X_{y,x}(t)\, \le \,  \sqrt{2A(y,x)   \log \Big( \frac{A(y,x)} {V(y,x)}\Big) }\Big\}\,\le \,    e^{-\,\frac{C \e\,V(y,x)}{  \sqrt{   ((\sum_{y\le k\le x} a_k^4)^{1/2}+1 ) \log ( {A(y,x)}/{V(y,x)})}}}.$$

  This achieves the proof.

   %%%%%%%%%%%%%%%%%%%%%%%%%%%%%%%%%%%%%%%%%%%%%%%%%%%%%%%%%%%%%%%%%%%%%%%%%%%%%%%%%%%%%%%%%%%%%%

  %In subsection \ref{sub.c.a.}, we extend these results to almost periodic Gaussian polynomials, by introducing a convenient   diophantine approximation.
  
\section{\bf Proof  of Theorem \ref{sup.kappa.th}.}\label{s4}   

   \vskip 2 pt 
   
 %Recall  that $\log \nu= \sum_{m=1}^\nu \frac1m -\g +\mathcal O(\frac1\nu)$, 
%  where $\g$ is Euler's constant.  Let 
%$H_\nu= \sum_{m=1}^\nu \frac1m$, $\nu\ge1$, be the $\nu$-th harmonic number. Thus % from \eqref{log.dev} that  
%$\log (\nu\,e^\g) = H_\nu  +\mathcal O(\frac1\nu)$. 
%These numbers are not integers for $\nu>1$. 
  %There is an  exact formula
%\begin{equation} H_\nu=\frac{1}{\nu!}\, \Big[\begin{matrix}\nu+2\cr2\end{matrix} \Big]  ,
%\end{equation}
%where $\big[\begin{matrix}\nu+2\cr2\end{matrix} \big]$ is a Stirling number of the first kind, which is  integer. Since the harmonic numbers have a logarithmic growth, it would be interesting to know more on the arithmetical structure of these kind of Stirling numbers.

 \subsection{Controlling approximation} We begin with explaining   the method used. For any  real $U\ge 1$,   \begin{equation*}   \sup_{1\le u\le U} X_{y,x}(u)\ge  \sup_{1\le u\le U} X^\perp_{y,x}(u) -
 \sup_{1\le u\le U} \big|X_{y,x}(u)-  X^\perp_{y,x}(u)\big|
.
  \end{equation*} 
  Let $\Theta_{y,x} >h>0$. We deduce that
  %from \eqref{control1} that
 \begin{align}\label{control1}  \P\Big\{   \sup_{1\le u\le U} &X_{y,x}(u)\le \Theta_{y,x} -h  \Big\}      
   \cr &\ \le\P\Big\{ \sup_{1\le u\le U} X^\perp_{y,x}(u) \le \Theta_{y,x}   \Big\}+\P\Big\{     \sup_{1\le u\le U }  \big|X_{y,x}(u)-X^\perp_{y,x}(u)\big|    > h   
   \Big\},
  \end{align}
and thus have to bound from above   the supremum,
$$\sup_{1\le u\le U} \big|X_{y,x}(u)-  X^\perp_{y,x}(u)\big|. $$ 
%  difference between these two processes is bounded as follows: 
%Observe here that this supremum is a growing function of $U$, whereas the probability we study $\P\big\{ \sup_{1\le u\le U}  |X_{y,x}(u)  |\le \Theta\big\}$  decreases as  $U$ increases; hence a balance of $U$ is to be searched. 
 We sieve $u$ with the test sequence $\{N_\k, \k\ge 1\}$.   Let $u$ be such that 
  $$ N_{\k-1}\le u< N_\k  .$$ 

  We  study local extrema
$$\sup_{N_{\k-1}\le u\le N_\k}   \big|X_{y,x}(u)- X^\perp_{ y,x}(u)\big|$$
restricted to those $\k$ such that $N_\k \le U$.  
  Note    (see \eqref{sup.G1} in Remark \ref{Orlicz}) that  \begin{align}\label{bound.diff.X.Xperp.1a}\E \sup_{1\le N_\k \le U }& \sup_{N_{\k-1} \le u\le N_\k }\big|X_{y,x}(u)- X^\perp_{y,x}(u)\big|
\cr & \le \ C\, \sqrt{\log  \k([1,U]) }\sup_{1\le N_\k \le U }\E  \sup_{N_{\k-1} \le u\le N_\k }\big|X_{y,x}(u)-X^\perp_{y,x}(u)\big|.
\end{align}
 Therefore it   suffices  to bound  for each $1\le N_\k\le U $,
$$\E  \sup_{N_{\k-1} \le u\le N_\k }\big|X_{y,x}(u)-X^\perp_{y,x}(u)\big|.$$
For controlling the probability
$$\P\Big\{     \sup_{1\le u\le U }  \big|X_{y,x}(u)-X^\perp_{y,x}(u)\big|    > h   
   \Big\}$$
   the following inequality (Lemma \ref{si.sn}) is  appropriate: if $N(X)$ is a Gaussian  semi-norm such that $\P\{ N(X)<\infty\}>0$, then $\E N(X)<\infty$, and in fact there exists an absolute constant $K$ such that 
 $$\E  \exp\Big\{ {N(X)^2\over K(\E N(X))^2}\Big\}\le 2.$$ %Further 
%which we do now.  

\subsection{An intermediate approximation result.}   We prove   \begin{theorem}\label{bound1}  Assume that \eqref{Nk.k} holds.
%\beq \label{Nk.k} N_k\ge k, \qq \quad k\ge 1.\eeq %Let $N(U)= \#\{\k\ge 1: N_\k\le U\}$.
 \vskip 3 pt 

{\rm (1)} Let  $1\le y\le U$.     We have for     $x>y    $,  
 \begin{align*}
\E \sup_{1\le u\le U} \big|X_{y,x}(u)-  &X^\perp_{y,x}(u)\big|     \le C E\,\sqrt{\log \k ([y, U])}  , \end{align*}
  where,
     \begin{eqnarray*}  E\,&=&    y  \, \Big( \sum_{   y\le k\le x}\frac{1 }{N^2_k  }\Big)^{1/2}\Big( \sum_{   y\le k\le x}a_k^2\Big)^{1/2} +
             \sum_{ y\le  k<  \k\atop N_\k \le U}   |a_k|
  \cr &  &   +   \sup_{\k:\,y\le N_\k\le U }\ N_\k  \Big( \sum_{   \k\le k\le x} \frac{  1}{N_k^{2 } }      \Big)^{1/2}\,\Big(\sum_{   \k\le k\le x}   |a_k|^2    \Big)^{1/2}
.
 \end{eqnarray*}

  {\rm (2)} Let $1\le U\le y$. Then \begin{equation*}\E \sup_{1\le u\le U} \big|X_{y,x}(u)-   X^\perp_{y,x}(u)\big|     \le C E'\,\sqrt{\log \k([1, U])} 
    ,
   \end{equation*} 
with 
 \begin{eqnarray*} E\,'&=&\,U \Big( \sum_{   y\le k\le x}\frac{1 }{N^2_k  }\Big)^{1/2}\Big( \sum_{   y\le k\le x}a_k^2\Big)^{1/2} .
 \end{eqnarray*} \end{theorem}
\begin{proof}  %Let   $1\le u\le U$. 
Using the elementary bound  $\max\big(|\cos a -\cos b|, |\sin a -\sin b|\big)\le 2|\sin\big(\frac{a-b}{2}\big)|\le 2\min(1,  |a-b|  )$,  and   \eqref{approxLk}     we get,
\begin{eqnarray}\label{start}&&\big|X_{y,x}(u)- X^\perp_{ y,x}(u)\big|
\cr &\le& \sum_{y\le k\le x} |a_k| \Big( |g_k| \big|\cos L_ku-\cos ( \ell( k)   u)\big|+ |g'_k|\big|\sin L_k u-\sin  ( \ell( k)   u)\big|\Big)
\cr &\le& 2\sum_{y\le k\le x} |a_k| \big( |g_k|  + |g'_k| \big)\big|\sin \frac12\big(L_k  -\ell( k) \big)   u\big|  
\cr &\le&  \sum_{y\le k\le x} |a_k| \big( |g_k|  + |g'_k| \big)\min\big(1, |L_k  -\ell( k)|u \big)
\cr &\le&  \sum_{y\le k\le x} |a_k| \big( |g_k|  + |g'_k| \big)\min\big(1, \frac{|u|}{N_k} \big).\end{eqnarray}
%Hence
%\ben\label{start1}\sup_{0\le u\le  U}\big|X_{y,x}(u)- X^\perp_{ y,x}(u)\big|&\le& \sum_{y\le k\le x} |a_k| \big( |g_k|  + |g'_k| \big)\min\big(1, \frac{U}{N_k} \big)
% .
%\een
% \begin{eqnarray}\label{gds.perp} 
 %$    X_{y,x}^\perp(t)&=& \sum_{y\le k\le x} a_k \big( g_k \cos L_kt+ g'_k\sin L_k t\big), \qq \qq t\in \R ,
 %  \end{eqnarray}  %$$\Big|\sum_{ y\le  p\le x} a_p \Big( g_p \cos  \big( (\log p) u \big)  + g'_p\sin  \big(  (\log p) u  \big) \Big)-\sum_{ y\le  p\le x} a_p \Big( g_p \cos  \big( H_{\lfloor \frac{p}{\g}\rfloor} u \big)  + g'_p\sin  \big( H_{\lfloor \frac{p}{\g}\rfloor} u  \big) \Big) \Big|$$
   %The coming inequalities, up to inequality \eqref{control1}, are almost sure.  
Let $\xi_k=|g_k|    + |g'_k|$, $k\ge 1$.  Let  $\k\ge 2$ and   $ N_{\k-1}\le u\le N_\k  $. We deduce,
      \begin{eqnarray}\label{basic.bound.diff.X.Xperp}  \sup_{u\in [N_{\k-1}, N_{\k } [}\big|X_{y,x}(u)-X^\perp_{y,x}(u)\big| 
&\le&  C  \sum_{ y\le  k  \le x}  |a_k| \min\Big(1, \frac{N_\k}{N_k  } \Big) \xi_k
.   \end{eqnarray}
%  \begin{eqnarray}\label{basic.bound.diff.X.Xperp.}  \E\sup_{u\in [N_{\k-1}, N_{\k } [}\big|X_{y,x}(u)-X^\perp_{y,x}(u)\big| 
%&\le&  C  \sum_{ y\le  k  \le x}  |a_k| \min\Big(1, \frac{N_\k}{N_k  } \Big)  
%.   \end{eqnarray}
    
    Consider two cases. 
%  \vskip 8 pt \hskip 50 pt  (i) \quad $[N_{\k-1}, N_\k[\subset [y, (y\vee U)]$,  \hskip  40 pt   (ii)\quad $[N_{\k-1}, N_\k[\subset [1, (y\wedge U)]$. 
% \vskip 8 pt
 %\noi Hence,
%\vskip 8 pt  (A) : ($1\le y\le U$).  
%  \vskip 8 pt  (B) : ($1\le U\le y$). Then (iii)  $[N_{\k-1}, N_\k[\subset [1, U]$
 $$ \hbox{(A) : ($1\le y\le U$)} \qq\qquad \hbox{(B) : ($1\le U\le y$)}$$
  %$y\le u\le U$.
%We  proceed as follows:  
%\begin{align}\label{bound.diff.X.Xperp.1}  \sup_{y\le u\le U}\big|X_{y,x}(u)-&X^\perp_{y,x}(u)\big|  
%\, \le\,   \sup_{y\le \k\le U } \sup_{(\k-1)^b\le u\le \k^b\le U}\big|X_{y,x}(u)-X^\perp_{y,x}(u)\big| ,
%   \end{align}N_{\k-1}^b\le u\le N_\k^b
    \vskip 5 pt \underline{Case} (A): 
    \vskip 3 pt  (i) %$[N_{\k-1}, N_\k[\subset [y, U]$ 
   % \vskip 3 pt   (ii) $[N_{\k-1}, N_\k[\subset [1, y]$. 
   Let     $u\in [N_{\k-1}, N_\k[\subset [y, U]$. 
 We split the summation interval $[y, x]$ in two sub-intervals.  
  We operate a ceasura at height $\kappa$, writing  $\sum_{ y\le  k  \le x}=\sum_{ y\le  k<  \k}+ \sum_{   \k\le k\le x}$. 
   Then by   \eqref{basic.bound.diff.X.Xperp},
   %  \begin{eqnarray}\label{bound.diff.X.Xperp}\big|X_{y,x}(u)-X^\perp_{y,x}(u)\big| 
 % &\le&    C  \Big(\sum_{ y\le  k<  \k}   |a_k| \xi_k +\sum_{   \k\le k\le x}   |a_k|    \Big(\frac{N_\k}{N_k  }\Big)   \xi_k \Big) ,    \end{eqnarray}
% and so 
 \begin{equation}\label{bound.diff.X.Xperp.a}\sup_{N_{\k-1} \le u\le N_\k }\big|X_{y,x}(u)-X^\perp_{y,x}(u)\big| 
  \,\le\,     C \Big( \sum_{ y\le  k<  \k}   |a_k| \xi_k +\sum_{   \k\le k\le x}   |a_k|    \Big(\frac{N_\k}{N_k  }\Big) \xi_k \Big) .
   \end{equation}
         
   By applying Cauchy-Schwarz's inequality, the second sum in the right-term of  \eqref{bound.diff.X.Xperp.a} is bounded as follows  
  \begin{eqnarray*}\label{bound.diff.X.Xperp.} 
   \sum_{   \k\le k\le x}  |a_k|   \Big(\frac{N_\k}{N_k  }\Big) \xi_k  
       &\le& 
       C N_\k \,\Big(\sum_{   \k\le k\le x}   |a_k|^2    \Big)^{1/2}\Big(\sum_{   \k\le k\le x} \frac{  \xi_k^2}{N_k^{2  } }      \Big)^{1/2}.
   \end{eqnarray*}
   Thus 
\begin{eqnarray*} \E \sum_{   \k\le k\le x}  |a_k| \Big(\frac{N_\k}{N_k  }\Big) \xi_k
& \le& 
 C N_\k\,\Big(\sum_{   \k\le k\le x}   |a_k|^2      \Big)^{1/2}\E\Big(\sum_{   \k\le k\le x} \frac{  \xi_k^2}{N_k^{2  } }      \Big)^{1/2}
 \cr & \le&
   C N_\k\,\Big(\sum_{   \k\le k\le x}  |a_k|^2      \Big)^{1/2} \Big(\E\sum_{   \k\le k\le x} \frac{  \xi_k^2}{N_k^{2 } }      \Big)^{1/2}
 \cr & \le&
  C N_\k\,\Big(\sum_{   \k\le k\le x}   |a_k|^2    \Big)^{1/2} \Big( \sum_{   \k\le k\le x} \frac{  1}{N_k^{2 } }      \Big)^{1/2}
. 
 \end{eqnarray*}

   As to the first sum, we have
$$  \E\sum_{ y\le  k<  \k}   |a_k| \xi_k
\le C   \sum_{ y\le  k<  \k\atop N_\k \le U}   |a_k| .$$
So that, 
\begin{align*} 
\E  \sup_{N_{\k-1}\le u\le N_\k}\big|&X_{y,x}(u)- X^\perp_{y,x}(u)\big| 
  \cr \le&   \ C    \sum_{ y\le  k<  \k\atop N_\k \le U}    |a_k|+ C N_\k\,\Big(\sum_{   \k\le k\le x}   |a_k|^2     \Big)^{1/2} \Big( \sum_{   \k\le k\le x} \frac{  1}{N_k^{2 } }      \Big)^{1/2}
.
   \end{align*}
 
 Hence,
  \begin{align}\label{ }
   &\sup_{\k:\atop   [N_{\k-1},N_\k[\subset [y, U] } \E  \sup_{N_{\k-1}\le u\le N_\k}\big| X_{y,x}(u)- X^\perp_{y,x}(u)\big| 
   \cr\, \le&
     \ C      \sum_{ y\le  k<  \k\atop N_\k \le U}   |a_k|+ C \sup_{\k:\,y\le N_\k\le U }\ N_\k  \Big( \sum_{   \k\le k\le x} \frac{  1}{N_k^{2 } }      \Big)^{1/2}\,\Big(\sum_{   \k\le k\le x}   |a_k|^2    \Big)^{1/2}. 
  \end{align}
 \begin{remark}\label{Nk.exp} If $N_k=2^k$, $k\ge 2$, then
  %$\k_y(U)= \mathcal O(\log \frac{U}{y})$ and 
  $$N_\k  \Big( \sum_{   \k\le k\le x} \frac{  1}{N_k^{2 } }      \Big)^{1/2}= 2^\k  \Big( \sum_{   \k\le k\le x} 2^{-2k}     \Big)^{1/2}=\mathcal O(1).$$
 \end{remark}

 \vskip 5 pt   (ii): Let   $u\in [N_{\k-1}, N_\k[\subset [1,  y ]$. We use assumption \eqref{Nk.k}, namely $N_k\ge k$ for each $k$. Then for $k\ge y$,
  $$N_\k\le y\le k\le N_k.$$
 %$N_{\k-1}\le u\le N_\k\le U$ and $ 1\le  N_\k\le y$. Then $N_\k\le y\wedge U $. Further if \underline{$N_k\ge k$}, then for $  k\ge y $ we have $N_k\ge y\ge N_\k$, so that in this case by \eqref{start},
  By \eqref{basic.bound.diff.X.Xperp}   \begin{eqnarray}\label{bound.diff.X.Xperp.kappa.le.y}\big|X_{y,x}(u)-X^\perp_{y,x}(u)\big| 
  %&\le&     \sum_{   y\le k\le x}   |a_k|   \Big(\frac{ N_\k}{N_k  } \Big)  \xi_k
  \, \le \,C\,u     \sum_{   y\le k\le x}       \Big(\frac{|a_k| }{N_k  } \Big) \xi_k .
   \end{eqnarray}
Thus, \begin{eqnarray*} \E \sup_{u\in [N_{\k-1}, N_\k[  }\big|X_{y,x}(u)-X^\perp_{y,x}(u)\big| 
&   \le  &   C\,  N_\k \sum_{   y\le k\le x}     \Big(\frac{ |a_k| }{N_k  } \Big) \, \E\xi_k %\le C\,(y  \wedge U)     \sum_{   y\le k\le x}       \Big(\frac{|a_k| }{N_k  } \Big)
\cr &\le & C\,  N_\k \Big( \sum_{   y\le k\le x}\frac{1 }{N^2_k  }\Big)^{1/2}\Big( \sum_{   y\le k\le x}a_k^2\Big)^{1/2}   ,
   \end{eqnarray*}
  and
\begin{equation*} \sup_{ [N_{\k-1},N_\k[\subset [1,y]  }\E\sup_{N_{\k-1} \le u\le N_\k }\big|X_{y,x}(u)-X^\perp_{y,x}(u)\big| 
   \le     C\,y \Big( \sum_{   y\le k\le x}\frac{1 }{N^2_k  }\Big)^{1/2}\Big( \sum_{   y\le k\le x}a_k^2\Big)^{1/2}     .
   \end{equation*}

  Consequently,
\begin{eqnarray}
&&\sup_{1\le N_\k \le U }\E  \sup_{N_{\k-1} \le u\le N_\k }\big|X_{y,x}(u)-X^\perp_{y,x}(u)\big|
\cr &\le &
 \Big(  \sup_{[N_{\k-1},N_\k[\subset [1,y] }+ \sup_{[N_{\k-1},N_\k[\subset [y,U]}\Big) \E \sup_{N_{\k-1} \le u\le N_\k }\E \big|X_{y,x}(u)-X^\perp_{y,x}(u)\big|
 \cr &\le &    C\,y \Big( \sum_{   y\le k\le x}\frac{1 }{N^2_k  }\Big)^{1/2}\Big( \sum_{   y\le k\le x}a_k^2\Big)^{1/2} +
       C     \sum_{ y\le  k<  \k\atop N_\k \le U}   |a_k|
  \cr &  &   + C \sup_{\k:\,y\le N_\k\le U }\ N_\k  \Big( \sum_{   \k\le k\le x} \frac{  1}{N_k^{2 } }      \Big)^{1/2}\,\Big(\sum_{   \k\le k\le x}   |a_k|^2    \Big)^{1/2}
. \end{eqnarray} 
In view of \eqref{bound.diff.X.Xperp.1a}, we deduce that   \begin{align*}
\E \sup_{1\le u\le U} \big|X_{y,x}(u)-  &X^\perp_{y,x}(u)\big|     \le C E\,\sqrt{\log \k([1,U])}  , \end{align*}
  with 
 \begin{eqnarray*}  E\,&=&    y  \, \Big( \sum_{   y\le k\le x}\frac{1 }{N^2_k  }\Big)^{1/2}\Big( \sum_{   y\le k\le x}a_k^2\Big)^{1/2} +
             \sum_{ y\le  k<  \k\atop N_\k \le U}   |a_k|
  \cr &  &   +   \sup_{\k:\,y\le N_\k\le U }\ N_\k  \Big( \sum_{   \k\le k\le x} \frac{  1}{N_k^{2 } }      \Big)^{1/2}\,\Big(\sum_{   \k\le k\le x}   |a_k|^2    \Big)^{1/2}
.
 \end{eqnarray*}

 \vskip 5 pt \underline{Case} (B):    
   % \vskip 3 pt$$ \hbox{(B) : ($1\le U\le y$)}$$
Let $u\in [N_{\k-1}, N_\k[\subset [1, U]$. %Here $1\le U\le y$, we only have to consider the case $[N_{\k-1}, N_\k[\subset [1, U]$. 
 By \eqref{basic.bound.diff.X.Xperp},   \begin{eqnarray*} \big|X_{y,x}(u)-X^\perp_{y,x}(u)\big| 
  %&\le&     \sum_{   y\le k\le x}   |a_k|   \Big(\frac{ N_\k}{N_k  } \Big)  \xi_k
  \, \le \,C  \sum_{ y\le  k  \le x}  |a_k| \min\Big(1, \frac{ u}{N_k  } \Big) \xi_k
   \, \le \,C \, N_\k\sum_{ y\le  k  \le x}    \frac{ |a_k|}{N_k  }  \, \xi_k .
   \end{eqnarray*}

Thus, \begin{eqnarray*} \E \sup_{N_{\k-1} \le u\le N_\k }\big|X_{y,x}(u)-X^\perp_{y,x}(u)\big| 
&   \le  &   C\, N_\k \sum_{   y\le k\le x}     \Big(\frac{ |a_k| }{N_k  } \Big) \, \E\xi_k %\le C\,(y  \wedge U)     \sum_{   y\le k\le x}       \Big(\frac{|a_k| }{N_k  } \Big)
\cr &\le & C\,  N_\k \Big( \sum_{   y\le k\le x}\frac{1 }{N^2_k  }\Big)^{1/2}\Big( \sum_{   y\le k\le x}a_k^2\Big)^{1/2}   ,
   \end{eqnarray*}
 and
\begin{equation*} \sup_{ [N_{\k-1},N_\k[\subset [1,U]  }\E\sup_{N_{\k-1} \le u\le N_\k }\big|X_{y,x}(u)-X^\perp_{y,x}(u)\big| 
   \le     C\,U \Big( \sum_{   y\le k\le x}\frac{1 }{N^2_k  }\Big)^{1/2}\Big( \sum_{   y\le k\le x}a_k^2\Big)^{1/2}     .
   \end{equation*}  
 If $N_k=2^k$, $k\ge 2$, then
  %$\k_y(U)= \mathcal O(\log \frac{U}{y})$ and 
  $$U \Big( \sum_{   y\le k\le x}\frac{1 }{N^2_k  }\Big)^{1/2}= U  \Big( \sum_{  y\le k\le x} 2^{-2k}     \Big)^{1/2}= \mathcal O(U2^{-y})= \mathcal O(1),$$
 since $U\le y$.
 In view of \eqref{bound.diff.X.Xperp.1a} again, it follows that  we deduce that    
 \begin{equation*}\E \sup_{1\le u\le U} \big|X_{y,x}(u)-   X^\perp_{y,x}(u)\big|     \le C E'\, \sqrt{\log  \k([1,U])} 
    ,
   \end{equation*} 
%with $\k (U)=\#\{ \k:[N_{\k-1}, N_\k[\subset [1, U]\}$.
with \ben\label{}E'&=&\,U \Big( \sum_{   y\le k\le x}\frac{1 }{N^2_k  }\Big)^{1/2}\Big( \sum_{   y\le k\le x}a_k^2\Big)^{1/2} .
\een

  This proves Theorem \ref{bound1}.
\end{proof}\bigskip \par 
% Let $N_k=2^k$, $k\ge 1$. Then $\sum_{   \b\le k\le x}  N_k^{-2 }  \le C2^{-2\b}$, $\b$ any positive integer. This implies that 
 %\begin{eqnarray*}  \D&\le &  2  \Big( \sum_{   y\le k\le x}a_k^2\Big)^{1/2} +            \sum_{ y\le  k\le \k_y(U)}   |a_k|.\end{eqnarray*} 

 \subsection{Gauss\-ian semi-norms}   We use strong integrability properties  of Gauss\-ian semi-norms.  
See 
%Fernique  \cite{F}, inequality 0.34, also 
for instance \cite{We},  Theorem 10.2.2   for a proof.  This one relies upon  rotational invariance property of Gaussian laws. We refer  the interested reader to the pages 497--500 of \cite{We}   for a detailed   discussion.
 \begin{lemma}\label{si.sn} Let $(E, \mathcal{ E})$ be a measurable vector space. Let $(\O,\mathcal{B},\P)$ be a
probability space. Consider a Gauss\-ian vector $X\colon (\O,\mathcal{B},\P)\to (E,
\mathcal{ E})$. Let $N=(E, \mathcal{ E})\to \R^+$ be a measurable semi-norm
on $E$ and assume that $\P\{ N(X)<\infty\}>0$. Then $\E N(X)<\infty$
and   there exists an absolute constant $K$ such that 
\begin{equation}\label{strong.i.G} \E  \exp\Big\{ {N(X)^2\over K(\E N(X))^2}\Big\}\le 2.  
 \end{equation}
Further the moments of  $N(X)$ are all equivalent.
\end{lemma}
  \begin{remark}[Orlicz's like inequalities]\label{Orlicz}Let $L^G(\P)\subset L^0(\P)$ be the Orlicz space associated to the Young function $G(t)=\exp(t^2)-1$, 
% with Orlicz norm $\|.\|_G$ (see Section 2) associated to the Young function$G(t)=\exp(t^2)-1$
%.   Let$\p:\R\rightarrow\R^+$ be a Young function (convex, even and $ \p(0)=0 $,  $\lim_{x\rightarrow\infty}  \p(x) =\infty$)
 with associated Orlicz's norm
 $\|f\|_G =\inf\{\alpha >0\!:\!
 \E
G(|f|/\alpha)\le 1\}$. By   Lemma \ref{si.sn},
%  and definition of the $G$-Orlicz's norm, 
$\|f\|_G \le K \E f$, if $f=N(X)$. As $G(t)\ge t^2/2$, letting $L>0$, for $f  \in L^G(\P)$, $$\E
G( | f  |/L\E | f  |)\ge (1/2)\E (| f  |/L\E | f  |)^2= (1/2)  (\E | f  |^2/(L\E| f  |)^2 >2,$$
if $L$ is a small enough (absolute) constant. So that $\|f\|_G \ge K' \E f$. Thus there exist positive absolute constants $K_1   \le K_2$, such that  as soon as $\P\{ N(X)<\infty\}>0$,\begin{equation}\label{G} K_1\,\E N(X) \le \|N(X) \|_G \le K_2\, \E N(X).
 \end{equation}
 
 We use the following inequality: let $f_j\in L^G(\P)$, $j=1,\ldots,n$, then 
\begin{equation}\label{sup.G}   \big\|\sup_{1\le j\le n} |f_j|  \big\|_G \le 
 \left(\d \log n\right)^{1/ 2} 
 \sup_{1\le j\le n} \|  f_j    \|_G ,
 \end{equation}
where $\d= {2/ \log\, 2}$. %l'in\'egalit\'e suivante
%$$\forall n\ge 2,\ \forall f_1,\cdots,f_n, \qquad 
%||\sup_{i=1}^n|f_i|||_{\Psi_q}\le (\Delta.\log\
%n)^{ {1\over q}}.\sup_{i=1}^n||f_i||_{\Psi_q}, \leqno (2.2.10)$$
%o\`u $\Delta= {2\over \log\ 2}$.\par
%Indiquons-en bri\`evement la preuve. Il n'y a aucune restriction  \`a supposer que
Indeed there is no loss to assume  $\|f_i\|_{G}\le 1,$ for $j=1,\ldots,n$. Since $\d  \log \, n\ge 1$,  we find by applying  Jensen's inequality:
 $$
 \E \exp\Bigl(  {1\over (\d \log  n) }. \sup_{1\le j\le
 n}|f_j|^2\Bigr)   \le \Bigl(\E \exp\Bigl(  \sup_{1\le j\le
 n}|f_j|^2\Bigr)  \Bigr)^{  {1\over (\d.\log  n)
 }}
  $$$$ \le \Bigl(\sum_{1\le j\le
 n}\E \exp (  |f_j|^2 )   \Bigr)^{  {1\over (\d \log  n)
 }} \le  (2n)^{  {1\over (\d \log  n)
 }}=\exp\bigl( {\log (2n)\over 2\log n}.\log2\bigr)\le 2. $$ 
 Thus \eqref{sup.G} follows from the definition of the  Orlicz's norm  $\|.\|_{G}$. Further \eqref{G} and \eqref{sup.G} imply that there exists a positive absolute constant  $K_3 $, such that   if $X_1,\ldots, X_n$ are Gaussian vectors satisfying  $\P\{ N(X_j)<\infty\}>0$ for each $j$, then
 \begin{equation}\label{sup.G1}    \E \sup_{1\le j\le n}N(X_j)    \le K_3 
 \left(  \log n\right)^{1/ 2} 
 \sup_{1\le j\le n} \E N(X_j)     . 
 \end{equation}
 
 \end{remark}
  \vskip 3 pt
\begin{proof}[Proof of Theorem \ref{sup.kappa.th}]
  We deduce from Theorem \ref{bound1} and estimates   \eqref{G},  \eqref{sup.G},    
 \begin{eqnarray}\label{G1.sup.kappa} \Big\| \sup_{1\le u\le U }  \big|X_{y,x}(u)-X^\perp_{y,x}(u)\big|  \Big\|_G   &\le &  C   \,\D  \sqrt{ \log \k([1,U])}.\end{eqnarray}  Let  
 $$\tau = \Big\{   \ \sup_{1\le u\le U }  \big|X_{y,x}(u)-X^\perp_{y,x}(u)\big| \le  h
   \Big\}.$$ Further from \eqref{control1},
 \begin{eqnarray*}&& 
 \P\Big\{   \sup_{1\le u\le U} X_{y,x}(u)\le \Theta_{y,x} -h  \Big\}\cr &=&\P\Big\{ \big\{  \sup_{1\le u\le U} X_{y,x}(u)\le \Theta_{y,x}  -h\big\}\cap \tau\Big\} +\P\Big\{ \big\{  \sup_{1\le u\le U} X_{y,x}(u)\le \Theta_{y,x}  -h\big\}\cap \tau^c \Big\}
 \cr & \le &
\P\Big\{ \sup_{1\le u\le U} X^\perp_{y,x}(u) \le \Theta_{y,x}   \Big\} +
\P\Big\{     \sup_{1\le u\le U }  \big|X_{y,x}(u)-X^\perp_{y,x}(u)\big|    > h    
   \Big\}.
  \end{eqnarray*}
Therefore in view of  \eqref{G1.sup.kappa},   for some universal constant  $C\ge   K $, 
  \begin{eqnarray*} & &       \P\Big\{       \sup_{1\le u\le U }  \big|X_{y,x}(u)-X^\perp_{y,x}(u)\big|^2    > h^2   \Big\}
 \cr   &=  &  \P\Big\{   \exp\Big\{    \frac{\sup_{1\le u\le U }  \big|X_{y,x}(u)-X^\perp_{y,x}(u)\big|^2 }{C   \,\D^2    \log \k([1,U]) }   \Big\}   >  \exp\Big\{\frac{h^2 }{C   \,\D^2    \log \k([1,U]) } \Big\}   \Big\} \cr &\le  &  \P\Big\{   \exp\Big\{    \frac{\sup_{1\le u\le U }  \big|X_{y,x}(u)-X^\perp_{y,x}(u)\big|^2 }{K(\E \sup_{1\le u\le U }  \big|X_{y,x}(u)-X^\perp_{y,x}(u) )^2}   \Big\}   >  \exp\Big\{\frac{h^2 }{C   \,\D^2    \log \k([1,U])  } \Big\}   \Big\}  \cr &\le  & 2\, \exp\Big\{\frac{- h^2 }{C  \,\D^2    \log \k([1,U])} \Big\}     .
 \end{eqnarray*}
 Consequently, 
 \begin{equation*}  \P\Big\{   \sup_{1\le u\le U} X_{y,x}(u)\le \Theta_{y,x} -h  \Big\}   \,  \le \, 
\P\Big\{ \sup_{1\le u\le U} X^\perp_{y,x}(u) \le \Theta_{y,x}   \Big\} +2\, \exp\Big\{\frac{- h^2 }{C   \,\D^2    \log \k([1,U])} \Big\}.
  \end{equation*}
 \end{proof}
 
%  \subsection{Cyclic stationary   approximating Gaussian process}
%  Let $D_{y,x}={\rm lcm}\{N_k,y\le k\le x\}$ and write $D_{1,x}=D_{x}$.    Consider  the nearby $D_{y,x}$-cyclic stationary Gaussian process  
%  \begin{equation}X^\perp_{ y,x}(u) = \sum_{   y\le  k\le x} a_k  \big( g_k \cos ( \ell( k)    u) \big) +& g'_k\sin  ( \ell( k)   u)  \big),\qq  u\in\R.
% \end{equation}    
%  \begin{eqnarray}\label{gds.perp} 
%  $    X_{y,x}^\perp(t)&=& \sum_{y\le k\le x} a_k \big( g_k \cos L_kt+ g'_k\sin L_k t\big), \qq \qq t\in \R ,
  %  \end{eqnarray} 

 \section{\bf Almost periodic Gaussian 
polynomials with linearly independent frequencies along lattices}
%Proofs of Theorems \ref{t3},  \ref{ak.cos},  \ref{bound.max.cos.part.enonce}.
  \label{s6}
%We   study in Section \ref{s6} the asymptotic behavior of   almost periodic Gaussian 
%polynomials having frequencies linearly independent over $\Q$. 
Consider   the family $X=X_x=\{ X_x(t),t\in \R\}$ of $x$-parametrized Gaussian Dirichlet generalized polynomials \begin{eqnarray*}
%X(t)=
X_x(t) &=& \sum_{k\le x} a_k \big( g_k \cos\lambda_kt+ g'_k\sin \lambda_k t\big), \qq \qq t\in \R .
\end{eqnarray*}
Here $(\lambda_k)_{k\ge 1}$ is an increasing unbounded sequence of   irrational numbers, which is linearly independent over $\Q$.
\vskip 2 pt
No condition is imposed on the coefficient sequence $(a_k)_{ k\le x}$ except that it is non-vanishing.
\vskip 2 pt
%Further   $(a_k)_{k\ge 1}$ are real numbers such that \begin{equation}\label{A1}A_x:= \sum_{k\le x} a_k^2\, \uparrow \,\infty. \end{equation}
%, and $x$ is a positive integer.
 Consider the sub-process  
\begin{eqnarray}\label{px.almost.periodic.I.II}
   X(ja)=   X_x(ja ), \qq j\ge 0 ,\quad a>0.
\end{eqnarray}
%where $a$ is an arbitrary  positive real.  
  The decoupling coefficient of $X $ is the quantity 
 %defined as follows
\begin{eqnarray}\label{px.almost.periodic.n}
p(x,a) 
%=\sum_{j=0}^{n-1}\frac{|\E X_x(0)X_x(ja)|}{\E (X_x(0))^2}
&=&\frac{1}{A(x)}\sum_{j\ge 0} \big|\sum_{k\le x} a_k^2 \cos  \l_k ja \big|
.
\end{eqnarray}
 \begin{theorem} \label{t3} For any reals $a>0$, $x>0$,
% and let $p_x $ be the decoupling coefficient of $X=\{X_x(ja), j\ge 1\}$. We have 
\begin{eqnarray}\label{px.almost.periodic.infty}
p(x,a) =\infty
.
\end{eqnarray}
\end{theorem}  
\begin{theorem} \label{ak.cos}
   For any non-negative reals $\a_1, \ldots , \a_x$,  any arithmetic progression $\mathcal N$,   
  \begin{eqnarray}\limsup_{  \mathcal N\ni \nu\to\infty}  \big|\sum_{k\le x} \a_k  e^{2\pi \nu\l_k}    \big|&=& \sum_{k\le x}   \a_k
. 
\end{eqnarray}
\end{theorem}
  These results were announced in \cite{W3}, Remark 4.4.
\subsection{Localized versions of Kronecker's theorem.}
 Recall first a classical result (Hardy and  Wright \cite{HW},\, th.\,443), if $\nu_1, \ldots, \nu_k, 1$ are linearly independent, then the set of points 
$$\big(\{n\nu_1\}, \ldots, \{n\nu_k\}\big) $$
is dense in the unit cube. 
  Kahane and   Salem showed in \cite{KS} pp.\,175-177, that if $\l_1,\ldots,\l_k$ are real numbers enjoying the following property: among the $3^k$ numbers, defined modulo $2\pi$,
$$ \e_1\l_1 +  \ldots+ \e_k\l_k, \qq \quad \e_j =-1,0, 1$$
each of the numbers $0, \l_1,    \ldots, \l_k$ (modulo $2 \pi)$ appears only once, then there exists, for any $\a>0$, a $T= T(\a,\l_1,    \ldots, \l_k)$ such that, for any $c_1,\ldots, c_k$, and $x$
\begin{equation} \sup_{x<u<x+T\atop u\,\hbox{\sevenrm integer}}\big| c_1e^{i \l_1 u}+\ldots + c_ke^{i \l_k u}\big|\,\ge \,\Big( \frac12 -\a\Big) \big( |c_1|+ \ldots+ |c_k|\big).\end{equation}
This is closer to our investigations.
% We   refer to the Gonek and Montgomery's survey \cite{GM}.
       We use the following lattice localized version of Kronecker's theorem improving   the localized Kronecker Theorem established in   Weber \cite{W4}, and offering   an estimate from below of the number of restricted approximations, which is    new  in the study of Kronecker's theorem. 
 
%%%%%%%%
 \begin{theorem} \label{t1}  There   exists a positive   constant $C_o\!<\!1/4$ such that for any positive integer 
 %$M$, 
$\o $,   and reals $\l_1,\l_2,\ldots,\l_N $,
 %  then for 
any  positive real $h$ such that
   \begin{eqnarray}\label{Xi}    \Xi= \Xi( 
   %M,
   h,\o,\l_1,
\l_2,
\ldots,
\l_N):=
     \min_{{\nu_\ell \,\hbox{\sevenrm integers}\atop 0<  \sup_{ \ell }  |\nu_\ell|\le  
     %M(=
     6 \o  \log {N\o\over C_o} 
     %) 
   }}
   %\atop 
 %\nu_1\l_1+\ldots+\nu_N\l_N\neq 0}   
 \Big\| h   \sum_{1\le \ell\le
 N  }   \l_\ell 
 \nu_\ell
 \Big\|\, >0,
\end{eqnarray}
%$$  T > \frac{3}{2\P\{S_k=0\}^{N }\Xi} ,$$
 then for  any real $T>h$, $ T    >   {1 \over   \Xi}\, \big( { 4\,\o  \over   
      C_o} \sqrt{ \log \ {N\o\over C_o}   }\big)^{N} $,
      %{\color{blue} $%|I\cap h\N|=
  %    \frac{T}{h}  \,  > \,      \frac{1}{ \Xi \P\{S_k=0\}^{N}}$,  $\frac{T}{h}    > {3 \over \pi \Xi} \big( { 2\o\sqrt k \over   
   %   C_0} \big)^{N}$}
the following approximation properties hold: 
\vskip 2 pt 
{\bf  (i)} For   any   interval $I$ of length $T$ and any  reals $ 
\b_1,\ldots, 
\b_N  $, 
%there exists   $t\in I\cap h\N$,   such that  
\begin{equation}\label{(3.1)} \min_{t\in I\cap h\N}\ \max_{1 \le j\le N } \|  t\l_j-\b_j \|\le  {1\over \o}.  
\end{equation}
%\begin{equation}\label{(3.1)} \max_{1 \le j\le N } \|  t\l_j-\b_j \|\le {1\over \o}.  
%\end{equation}

{\bf  (ii)} For $N$ and $\o$ large
%\begin{equation}\label{(3.9.)}
%  \#\Big\{   t \in  I\cap h\N: \max_{1\le j\le N}\|  t\l_j-\b_j \|\le \frac{1}{\o}\Big\}\ge  \frac{1-o(1)}{8}\,e^{-N(\sqrt{3/2\pi} +2\o\sqrt k)}   \ \#\big\{I\cap h\N\big\} ,\end{equation}
 \begin{equation}\label{(3.9.)}
  \#\Big\{   t \in  I\cap h\N: \max_{1\le j\le N}\|  t\l_j-\b_j \|\le \frac{1}{\o}\Big\}\ge   \Big({C\over
\o\sqrt{
   k }}\Big)^{ N}  |I\cap h\N| ,\end{equation}
%and $ C_0$ is an absolute constant.
where 
\begin{equation}\label{enonce}  k=  \inf\Big \{j\ge 1:  {N\o\over
 C_0     }    \le  {4^{2j-1}\over \sqrt j}\Big\},  
\end{equation}
{\bf  (iii)} For some absolute constant $C$, 
\begin{equation*}\label{(3.9c)}
\#\Big\{   t \in  I\cap h\N: \max_{1\le j\le N}\|  t\l_j-\b_j \|\le {1/\o}\Big\}\ge    {C^{N/2} \over  h \Xi}   .\end{equation*} 
\end{theorem}

\begin{remarks}\label{rem.t1} (1)  In Claim (i)
%ssertion before \eqref{(3.1)} may look a bit confusing (also in similar statements..) since
 the set of  $(\b_1,\ldots, 
\b_N)$ has power of continuum, whereas $I\cap h\N$ is finite.   This may look   confusing  (also in similar statements \ldots), but  hides a simple fact, which is easy to  clarify (proceed with suitable finite covering of $I$). 
 
\vskip 2 pt  \noi (2)  Let $0<\b<\pi/2$ be fixed.   We may choose    data $\o\ge 1$ integer so that $\b\o\ll 1$. By selecting $\b_j\equiv \b$ in Claim (i), we get
for any real $T>h$, $ T    >   {1 \over   \Xi}\, \big( { 4\,\o  \over   
      C_0} \sqrt{ \log \ {N\o\over C_0}   }\big)^{N}  
      $,  any   interval $I$ of length $T$\begin{equation*}  \min_{t\in I\cap h\N}\ \max_{1 \le j\le N } \|  t\l_j-\b  \|  \le \frac1\o. 
\end{equation*}
%We can thus select $t$ in a lattice.
(See \cite{W4}, Remark\,1.2). 
\vskip 2 pt  \noi (3) Condition   $ \Xi>0$ is satisfied  if the reals $\l_1,\l_2,\ldots,\l_N $, are linearly independent.\end{remarks}
 
\subsection{Proofs of Theorems \ref{t3} and \ref{ak.cos}} 
  Let  $\|u\|$ denotes the distance from $u$  to the nearest integer, namely $\|u\|= \min_n|u-n|$.  
 %  Each $X_x(t)$  is  a stationary   Gaussian process on $\R$.
   % although its cosine part  and sine part  aren't. 
 We have for $1\le x\le y,\   s,t\in\R$, the correlation relations
\begin{eqnarray}\label{corr} \E X_x(s)X_y(t)&=& \sum_{k\le x} a_k^2 \cos\l_k(s-t),   \cr \E X_x(t)^2& =& A_x
%\cr \E\bigg( \frac{X_x(s) }{\sqrt{\E X_x(s)^2 }}\cdot \frac{X_y(t) }{\sqrt{\E X_y(t)^2 }}\bigg)& =&  \frac{\sum_{k\le x} a_k^2 \cos\l_k(s-t)}{\sqrt{ A_x   A_y }} 
.
\end{eqnarray}

   Let $\o$ be some positive integer.
    Let $h$ also be  a positive real, and $I$   an  interval of length $T>h$, where $T $ is as in Theorem   \ref{t1} with $N=x$, and so $T=T(x,\o,\Xi)$. 
   % We choose $\b_k$ so that $0<\b_k<\pi $ and $
%\min_k \cos \b_k\ge {1/\e \o} $. 
Let $  
\b_1,\ldots, 
\b_x  $ be reals.  Then there exists   $\nu\in I\cap h\N$,   such that  
\begin{equation}\label{proof.t1a.}  \max_{1 \le k\le x } \|  \nu\l_k-\b_k\|\le {1\over \o}.  
\end{equation} 
As 
$$|e^{ i  \nu\l_k}-e^{ i  \b_k}|  =2  |\sin     \frac12 (  \nu\l_k-\b_k)| \le 2   \|  \nu\l_k-\b_k\|\le {2   \over \o},$$
we have 
\beq \label{approx.t1a}\max_{1 \le k\le x }\big|e^{2i  \nu\l_k}-e^{2i  \b_k}\big| \le 
{2   \over \o}.
\eeq
 
 %$$|e^{i\pi \nu\l_k}-e^{i\pi \b_k}|^2= \big(\cos(\pi \nu\l_k)-\cos(\pi \b_k)\big)^2+\big(\sin(\pi \nu\l_k)-\sin(\pi \b_k)\big)^2
 %$$  
  
Let $0<\e<1$ be a fixed real, and  suppose   $\o$ chosen so  that $  \o > (2  /\e)$.  Let also $\b_k\equiv \b$. Then for any sequence  $\{a_k, 1\le k\le x\}$ of non-negative reals,
$$
 \Big|\big|\sum_{k\le x} a_k   e^{2i  \nu\l_k}\big|- \sum_{k\le x} a_k     \Big|=\Big|\big|\sum_{k\le x} a_k   e^{2i  \nu\l_k}\big|-\big|e^{2i  \b } \sum_{k\le x} a_k \big|    \Big|\le \e\,\sum_{k\le x}  a_k. 
$$

Consequently, for any interval $I$ of length $T$ there exists   $\nu\in I\cap h\N$, such that
$$
 \big|\sum_{k\le x} a_k   e^{2i  \nu\l_k}\big| \ge (1-\e)  \sum_{k\le x}  a_k . 
$$
 But $h$ and $\e$ are arbitrary. We therefore have proved  that 
given any reals $a_1, \ldots , a_x$, and any arithmetic progression $\mathcal N$, 
$$
\limsup_{  \mathcal N\ni \nu\to\infty}  \big|\sum_{k\le x} a_k  e^{2i  \nu\l_k}    \big|= \sum_{k\le x}  a_k, 
$$
 
Theorem \ref{t3}  now follows from   Theorem \ref{ak.cos}.

%%%%%%%%%%%
%%%%%%%%%%%
%\subsection{Proof  of Theorem \ref{bound.max.cos.part.enonce}.}
\subsection{\bf An additional correlation result.}  
Now let   $ I_u= [u T, (u+1)T]$,   $u=1,\ldots, m$, where  $T $ is as in Theorem   \ref{t1},  and let   $  j_u\in I_u\cap h\N$,  be such that \eqref{approx.t1a} holds with  $\nu=j_u$ for each $u$, namely we have  
\beq \label{approx.ju}\max_{1 \le u\le m }\max_{1 \le k\le x }\big|e^{2i  j_u\l_k}-e^{2i  \b_k}\big| \le 
{2   \over \o}.
\eeq 
 For $\b_k\equiv \b>0$ small, we have
$$   \sum_{k\le x} a_k^2   \cos^2\l_kj_u a \sim   \sum_{k\le x} a_k^2, \qq \quad u=1,\ldots, m. $$
  We   consider the Gaussian sequence $\{X_x^{\rm cos}(j_u a),1\le u\le m \}$, where $X_x^{\rm cos}(t)=\sum_{k\le x} a_k   g_k \cos\l_kt $.   
   \begin{proposition}\label{cov.cos.part.eta}
  Let   $X_x^{\rm cos}(t)=\sum_{k\le x} a_k   g_k \cos\l_kt $, $t\in\R$. For some positive real  $ \eta<1$, 
 \ben  \E\bigg( \frac{X_x^{\rm cos}(j_u a)}{\sqrt{\E X_x^{\rm cos}(j_u a)^2 }}\cdot \frac{X_x^{\rm cos} (j_v a) }{\sqrt{\E X_x^{\rm cos} (j_v a)^2 }}\bigg)
&  \le &   \eta  ,
  \een    
   for all $u,v=1,\ldots, m$, $u\neq v$. 
    \end{proposition}
 One can take $\eta= 1-{2   / \o}   
  $,  provided that     $ \frac{c}{2}  \b ^2<1$, and   $\o> {12\pi  / c(\pi \b)^2 } $,  $ c  $ being fixed in $]0,({2}/{\pi})[$).  
  See Remark \ref{Rem.Ac1a.}.   We don't know what can be done for the initial process  $X_x (t)  $.

 \begin{proof} We have the obvious correlation relations: 
 \begin{eqnarray}\label{corr.X.cos}  \E  X_x^{\rm cos}(s)X_x^{\rm cos}(t)&= &\sum_{k\le x} a_k^2 \cos\l_ks\cos\l_kt,  
  \cr \E X_x^{\rm cos}(t)^2& =& \sum_{k\le x} a_k^2   \cos^2\l_kt  
  \cr \E\bigg( \frac{X_x^{\rm cos}(s)}{\sqrt{\E X_x^{\rm cos}(s)^2 }}\cdot \frac{X_x^{\rm cos} (t) }{\sqrt{\E X_x^{\rm cos} (t)^2 }}\bigg)& =&  \frac{\sum_{k\le x} a_k^2 \cos\l_ks\cos\l_kt}{\sqrt{\displaystyle{\sum_{k\le x} a_k^2   \cos^2\l_ks    }}\sqrt{\displaystyle{ \sum_{k\le x} a_k^2   \cos^2\l_kt    }}}.
\end{eqnarray} As 
 \begin{align*} \cos ( \l_k  &j_u a)  \cos ( \l_k  j_v a) - \cos^2 (  \b_k)
\cr &=\big(\cos ( \l_k  j_u a)-\cos (  \b_k)\big)\cos ( \l_k  j_v a)+\big(\cos ( \l_k  j_v a)-\cos (  \b_k)\big) \cos (  \b_k),\end{align*}
using \eqref{proof.t1a.},
we get  
 \begin{eqnarray}\label{cos.part.bound}\big|\cos ( \l_k  j_u a) \cos ( \l_k  j_v a)- \cos^2 (  \b_k)\big|&\le &{2 \big( |\cos ( \l_k  j_v a)|+|\cos (  \b_k)|\big) \over \o}.
\end{eqnarray} 
 
 %By  Remark \ref{rem.t1}-(2), $0<\b<\pi/2$ be fixed in advance, we  can choose  data $\o\ge 1$ integer at convenience so that $\b\o\ll 1$.     
 
% We get
%for any real $T>h$, $ T    >   {1 \over   \Xi}\, \big( { 4\,\o  \over   
  %    C_0} \sqrt{ \log \ {N\o\over C_0}   }\big)^{N}  
   %   $,  any   interval $I$ of length $T$\begin{equation*}  \min_{t\in I\cap h\N}\ \max_{1 \le j\le N } \|  t\l_j-\b  \|    \ll \b  
%\end{equation*}and so 
Letting $\b_k\equiv \b $  we get the bound,
%$$\E  X_x^{\rm cos}(j_u a)X_x^{\rm cos}(j_v a)-\sum_{1\le k\le x}a_k^2 \cos^2 (\pi \b_k)$$
\begin{eqnarray}&&\big|\E  X_x^{\rm cos}(j_u a)X_x^{\rm cos}(j_v a)-\sum_{1\le k\le x}a_k^2 \cos^2 (  \b )\big|
\cr &=&\Big| \sum_{1\le k\le x}a_k^2\Big(\cos ( \l_k  j_u a) \cos ( \l_k  j_v a)- \cos^2 (  \b)\Big)\Big|
\cr &\le &{2  ( |\cos ( \l_k  j_v a)|+|\cos (  \b )|) \over \o}\Big( \sum_{1\le k\le x}a_k^2 \Big)
%\ \le \  {4\pi\over \o} \sum_{1\le k\le x}a_k^2
 .
\end{eqnarray}
Thus
\begin{equation}  \E  X_x^{\rm cos}(j_u a)X_x^{\rm cos}(j_v a)   \,\le \,\Big(\cos^2 (  \b)+{2 ( 1+|\cos (  \b )|) \over \o}\Big)\Big( \sum_{1\le k\le x}a_k^2 \Big)
%\ \le \  {4\pi\over \o} \sum_{1\le k\le x}a_k^2
 .
\end{equation}
 
But $\ |e^{  \l_kj_u a}-e^{    \b } |=  |e^{  (\l_kj_u a-\b)}-1 |\ge |\sin (\l_kj_u a-\b)|$, so that
 \beq\label{est.square}\sum_{k\le x} a_k^2   \cos^2\l_kj_u a=\sum_{k\le x} a_k^2- \sum_{k\le x} a_k^2    \sin^2\l_kj_u a \ge \big(1-{2   / \o} \big) \sum_{k\le x} a_k^2. 
 \eeq
  Similarly,
  $\sum_{k\le x} a_k^2   \cos^2\l_kj_v a
 =\sum_{k\le x} a_k^2- \sum_{k\le x} a_k^2    \sin^2\l_kj_v a
  \ge \big(1-{2   / \o} \big) \sum_{k\le x} a_k^2 
   $. 
Therefore
  \ben \label{bound.cov.cos} \E\bigg( \frac{X_x^{\rm cos}(j_u a)}{\sqrt{\E X_x^{\rm cos}(j_u a)^2 }}\cdot \frac{X_x^{\rm cos} (j_v a) }{\sqrt{\E X_x^{\rm cos} (j_v a)^2 }}\bigg)
&  \le & \frac{(\cos^2 (  \b) +
     {4 ( 1+|\cos (  \b )|) \over \o})\big(\sum_{1\le k\le x}a_k^2\big)}{\big(1-{2   / \o} \big) \sum_{k\le x} a_k^2 }
 \cr &\le &  \frac{(\cos^2 (  \b) +{4   / \o})\big(\sum_{1\le k\le x}a_k^2\big)}{\big(1-{2   / \o} \big) \big(\sum_{1\le k\le x}a_k^2\big) }
  \cr &=& \frac{(\cos^2 (  \b) +{4   / \o}) }{\big(1-{2   / \o} \big)  }.
  \een

  Recall that $\sin x \ge ({2}/{\pi})x $  if $0\le x\le ({\pi}/{2})$.  Let $f(x)= \cos x -1+cx^2/2$, where $0<c<({2}/{\pi})$. Then $f'(x) =-\sin  x +cx \le ( c- ({2}/{\pi}) )x\le 0$, if $0\le x \le  ({\pi}/{2})$. Thus $$ \cos x -1+cx^2/2\le 0, \qq\qq 0\le x \le  ({\pi}/{2}).$$
  
   It follows that for $0\le \b \le  ({\pi}/{2})$,
 $ \cos   \b  \le \big(1 -c \big (\frac{   \b ^2}{2}\big)\big)$.   
Hence
$$ \cos^2  \b  \le \Big(1 -c \big (\frac{  \b ^2}{2}\big)\Big)^2 =1 +   \frac{c^2  \b ^4}{4}-   c \b ^2 \le 1      -   \frac{c}{2}  \b ^2 ,  $$
if $\frac{c^2  \b ^4}{4}\le \frac{c}{2} \b ^2$, namely 
 $ \frac{c}{2} \b ^2<1$,  which holds whenever $\b $ is small enough. Thus$$ \cos^2   \b  +{4   / \o}\le 1      -   \frac{c}{2}  \b ^2 +{4  / \o} <  1-{2   / \o} ,   $$
   if $ {6  / \o} <\frac{c}{2}   \b ^2$.  We were free to select in advance $\o$. It suffices that $\o> {12   / c \b ^2 } $.
 Therefore for some positive real  $ \eta<1$
\beq\label{beta.omega}\frac{(\cos^2   \b  +{4   / \o}) }{\big(1-{2   / \o} \big)  }<  1-{2   / \o}\,\le \,\eta,
\eeq
and so by \eqref{bound.cov.cos},  for $u,v=1,\ldots, m$, $u\neq v$,
 \ben  \E\bigg( \frac{X_x^{\rm cos}(j_u a)}{\sqrt{\E X_x^{\rm cos}(j_u a)^2 }}\cdot \frac{X_x^{\rm cos} (j_v a) }{\sqrt{\E X_x^{\rm cos} (j_v a)^2 }}\bigg)
&  \le &   \eta  .
  \een 
 
 \end{proof}
  
\begin{remarks}\label{Rem.Ac1a.}
%[value of $\eta$]
(i) By recapitulating:   $ c  $ is fixed in $]0,({2}/{\pi})[$, next $\b$ is chosen   so that $ \frac{c}{2}  \b ^2<1$, and after $\o$ so that $\o> {12\pi  / c(\pi \b)^2 } $. Then \eqref{beta.omega} holds. Thus  
  $\eta= 1-{2   / \o}   
  $
is suitable. 
%As ${ 2\pi  /  \o } < c(\pi \b)^2   /6$, we also have $   1-{ 2\pi  /  \o } >1- c(\pi \b)^2   /6 $, and so by \eqref{est.square} and \eqref{beta.omega},
%  \beq\label{est.square.a}\sum_{k\le x} a_k^2   \cos^2\l_kj_u a  \ge \big(1-{2  \pi/ \o} \big) \sum_{k\le x} a_k^2  \ge \eta \sum_{k\le x} a_k^2 . 
% \eeq
\vskip 2 pt (ii) Further by \eqref{est.square} and \eqref{beta.omega},
  \beq\label{est.square.a} \E\, X _x^{\rm cos}(j_u a)^2 = \sum_{k\le x} a_k^2   \cos^2\l_kj_u a  \ge \big(1-{2   / \o} \big) \sum_{k\le x} a_k^2  = \eta \sum_{k\le x} a_k^2 . 
 \eeq
 \end{remarks}

 \vskip 4 pt 
Let      $\eta$ be as   in Remark \ref{Rem.Ac1a.}-(i), let also  $Y=\{Y_u,1\le u\le m\}$ be the Gaussian vector having covariance function defined by
\begin{equation*}\begin{cases} \E Y^2_u=1 &\quad u=1,\ldots, m\cr 
\E Y_uY_v=\eta &\quad  u,v=1,\ldots,m, u\neq v .
\end{cases}\end{equation*}
 
  Put 
$$\widetilde{ X}_x (t)= \frac{X_x^{\rm cos}(t)}{\sqrt{\E X_x^{\rm cos}(t)^2 }}\qq \quad t\in\R.$$

 By Proposition \ref{cov.cos.part.eta}, 
\begin{eqnarray*} \E \widetilde{ X}_x^2 (   j_u a) &=&\E Y^2_u \qquad \ \ \, u=1,\ldots, m,\cr \E \widetilde{ X}_x(    j_u a)\widetilde{ X}_x(    j_v a)&\le &\E Y_uY_v \qquad  u,v=1,\ldots,m, u\neq v .\end{eqnarray*} 
 It follows from Theorem \ref{Ac1c} that  for any positive real $h$,
 \begin{eqnarray}\label{bound.max.cos.part}  \P\Big\{\sup_{u=1}^m\widetilde{ X}_x(\pi   j_u a) <h\Big\}&\le  &\P\Big\{\sup_{u=1}^mY_u <h\Big\} .
 \end{eqnarray} 
 %     \begin{lemma}[\cite{W0},\,Th.\,1.1]\label{Ac1a} Let $0\le \l <1$. Let also $X=(X_1,\ldots X_n)$ be a centered Gaussian vector  such that $\E X_i^2=1$, for each $j$, $\E X_iX_j= \l$, if $i\neq j$. For any positive real   $\Theta$,  
 % \begin{eqnarray*}\P\Big\{ \sup_{i=1}^n X_i\le \Theta\Big\}  &\le &\Big(1+\frac{  \l n  }{   1-\l }\Big)^{(n-1)/2}  \ \Phi \Big(\frac{\Theta}{\sqrt{ 1+\l(n-1) }}\Big)^n   .
%\end{eqnarray*}
% \end{lemma}   
 
 As a direct consequence of   Remark \ref{Rem.Ac1a.}-(ii) %and Lemma \ref{Ac1a},
%   As a direct consequence of   Remark \ref{Rem.Ac1a.}-(ii) and \cite[Cor.\,2.6]{W6}  
we have  for any $\kappa>0$,
   \beq\label{Ac1b}
   \P\Big\{
   \sup_{u=1}^m  X^{{\rm cos}}_x(\pi   j_u a)
    < \eta
     \big(\sum_{k\le x} a_k^2 \big)^{1/2}\kappa\Big\} % &=& \P\Big\{\sup_{u=1}^m\widetilde{ X} (\pi   j_u a) <\k\Big\}  
     \,\le \,\frac{1 }{  (1-\eta)^{(m-1)/2} }\ \Phi \Big(
     \frac{\kappa}{\sqrt{ 1+\eta(m-1) }}\Big)^m 
 %  \cr  &= &\frac{1 }{  (1-\eta)^{(m-1)/2} }\ e^{m\log  (1-\Psi  (\frac{\k}{\sqrt{ 1+\eta(m-1) }} ) )} 
% \cr  &\le &\frac{1 }{  (1-\eta)^{(m-1)/2} }\ e^{-m \Psi  (\frac{\k}{\sqrt{ 1+\eta(m-1) }} )  } 
      .
\eeq

 %%%%%%%%%
%%%%%%%%%
%%%%%%%%%

\end{document}